\documentclass[12pt]{amsart}
\usepackage[T1]{fontenc}
\usepackage{amsfonts}
\usepackage{mathrsfs}
\usepackage{amscd,amsmath,amssymb,amsfonts}

\usepackage[a4paper,text={140mm,240mm},centering,headsep=5mm,footskip=10mm]{geometry}

\usepackage[all]{xy}
\theoremstyle{plain}
\newtheorem{thm}{Theorem}
\newtheorem{lem}[thm]{Lemma}
\newtheorem{cor}[thm]{Corollary}
\newtheorem{prop}[thm]{Proposition}
\newtheorem{conj}[thm]{Conjecture}
\theoremstyle{definition}
\newtheorem{defn}[thm]{Definition}
\newtheorem{defns}[thm]{Definitions}

\newtheorem{rmk}[thm]{Remark}
\newtheorem{rmks}[thm]{Remarks}

\newtheorem{nota}[thm]{Notation}

\newtheorem{claim}[thm]{Claim}

\numberwithin{thm}{section} \numberwithin{equation}{section}

\newcommand{\p}{\partial}

\newcommand{\eq}[2]{\begin{equation}\label{#1}#2 \end{equation}}
\newcommand{\ml}[2]{\begin{multline}\label{#1}#2 \end{multline}}
\newcommand{\ga}[2]{\begin{gather}\label{#1}#2 \end{gather}}
\newcommand{\Exp}{{\rm Exp}}
\newcommand{\gr}{{\rm gr}}

\newcommand{\inj}{\hookrightarrow}

\newcommand{\Pic}{{\rm Pic}}

\newcommand{\Hom}{{\rm Hom}}

\newcommand{\Spec}{{\rm Spec \,}}

\newcommand{\Char}{{\rm char}}

\newcommand{\sC}{{\mathcal C}}

\newcommand{\sF}{{\mathcal F}}

\newcommand{\sH}{{\mathcal H}}

\newcommand{\sK}{{\mathcal K}}

\newcommand{\sO}{{\mathcal O}}

\newcommand{\A}{{\mathbb A}}

\newcommand{\C}{{\rm  C}}    
\newcommand{\D}{{\rm D}}    

\newcommand{\G}{{\mathbb G}}

\newcommand{\N}{{\mathbb N}}
\renewcommand{\P}{{\mathbb P}}
\newcommand{\Q}{{\mathbb Q}}

\newcommand{\Z}{{\mathbb Z}}
\newcommand{\bS}{{\mathbb S}}

\newcommand{\et}{{\rm \acute{e}t}}

\newcommand{\nn}{{\newline\noindent }}

\newcommand{\pb}[1]{\langle #1 \rangle}

\newcommand{\dR}{{\rm dR}}
\newcommand{\bull}{\centerdot}
\newcommand{\Nis}{{\rm Nis}}
\newcommand{\Sm}{{\rm Sm}}
\newcommand{\Sch}{{\rm Sch}}
\newcommand{\st}{{\rm pro}}
\newcommand{\Sh}{{\rm Sh}}
\newcommand{\etNis}{{\et/\Nis}}
\newcommand{\pro}{{\rm pro}}
\newcommand{\cont}{{\rm cont}}

\newcommand{\cone}{{\rm cone}}
\newcommand{\cris}{{\rm cris}}
\newcommand{\bulle}{{\scriptstyle \bullet}}
\renewcommand{\S}{{\rm S}}

\newcommand{\CH}{{\rm CH}}
\newcommand{\BGL}{{\rm BGL}}
\newcommand{\ch}{{\rm ch}}
\newcommand{\rK}{{\rm K}}
\newcommand{\red}{0}
\newcommand{\ve}{{\varepsilon}}

\begin{document}

\title[$p$-adic deformation]{$p$-adic deformation of  algebraic cycle classes}
\author{Spencer Bloch}
\author{H{\'e}l{\`e}ne Esnault}
\author{Moritz Kerz}
\address{5765 S. Blackstone Ave., Chicago, IL 60637,
USA}
\address{ Fakult\"at f\"ur Mathematik, Universit{\"a}t Duisburg-Essen, 45117 Essen, Germany}
\address{Fakult\"at f\"ur Mathematik, Universit\"at Regensburg, 93040 Regensburg, Germany}
\email{spencer\_bloch@yahoo.com }
\email{esnault@uni-due.de}
\email{moritz.kerz@mathematik.uni-regensburg.de}
\date{October 19, 2012}
\thanks{The second author is supported by  the SFB/TR45,
the ERC Advanced Grant 226257 and the Chaire d'Excellence 2011, the third by the DFG
Emmy-Noether Nachwuchsgruppe ``Arithmetik \"uber endlich erzeugten K\"orpern''
.}
\begin{abstract} 
We study the $p$-adic deformation properties of algebraic cycle classes modulo
rational equivalence. 
We show that the crystalline Chern character of a vector bundle lies
in a certain part of the Hodge filtration if and only if, rationally,
the class of the vector bundle lifts to a formal pro-class in
$K$-theory on the $p$-adic scheme.
\end{abstract}

\maketitle
\begin{quote}

\end{quote}
\section{Introduction}

\noindent
In this note we study the deformation properties of algebraic cycle classes modulo rational
equivalence. In the end the main motivation for this is to construct new interesting
algebraic cycles out of known ones by means of a suitable deformation process.
In fact we suggest that one should divide such a construction into two steps: Firstly, one
should study formal deformations to infinitesimal thickenings and secondly, one should try
to algebraize these formal deformations.

We consider  the first problem of formal deformation in the
special situation of deformation of cycles in the $p$-adic direction for a scheme over a
complete $p$-adic discrete valuation ring. It turns
out that this part is -- suitably interpreted -- of a deep cohomological and $K$-theoretic
nature, related to $p$-adic Hodge theory, while the precise geometry of the varieties plays only a minor r\^ole.

In order to motivate our approach to the formal deformation of algebraic cycles we start
with the earliest observation of the kind we have in mind,  which is due to Grothendieck.
The deformation of the Picard group can be described in terms of Hodge theoretic data
via the first Chern class.

Indeed, consider a field $k$ of characteristic zero, $S= k[[t]]$, $X/S$ a smooth projective
variety and $X_n\hookrightarrow X$ the closed immersion defined by the ideal $(t^n)$.  
The Gau{\ss}-Manin connection
\ga{}{     \nabla: H^i_{\dR}(X/S)\to  \hat{\Omega}^1_{S/k} \hat \otimes    H^i_{\dR}(X/S)    \notag}
is trivializable over $S$ by \cite[Prop.~8.9]{Katz}, yielding an isomorphism from the horizontal de Rham classes
over $S$ to de Rham classes over $k$
\ga{}{ \Phi: H^i_{\dR}(X/S)^\nabla \xrightarrow{\sim} H^i_{\dR}(X_1/k) .  \notag }
An important property, which is central to this article, is that $\Phi$ does not induce an isomorphism of the Hodge filtrations
\ga{}{      H^i_{\dR}(X/S)^\nabla\cap  F^r H^i_{\dR}(X/S)
 \xrightarrow{\nsim}  F^rH^i_{\dR}(X_1/k)                 \notag}
in general.
This Hodge theoretic property of the map $\Phi$ relates to the exact obstruction sequence
\ga{}{     \Pic(X_n)\to \Pic(X_{n-1})\xrightarrow{{\rm Ob}}  H^2(X_1, \sO_{X_1})      \notag} 
via the first Chern class in de Rham cohomology, see \cite{Blo}.

These observations produce a proof for line bundles of the following version of Grothendieck's variational Hodge conjecture \cite[p.~103]{GdR}.
\begin{conj}\label{into.conjhodge}
For $\xi_1\in K_0(X_1)_{\Q} $ such that $$\Phi^{-1}\circ \ch (\xi_1)\: \in\: \bigoplus_r
F^rH^{2r}_{\dR}(X/S),$$ 
there is  a $\xi   \in K_0(X)_{\Q}$, such that $\ch(\xi|_{X_1})=\ch(\xi_1)\in \bigoplus_r H^{2r}_{\dR}(X_1/k). $ Here $\ch$ is the Chern character. 
\end{conj}

In fact, using Deligne's ``partie fixe'' \cite[Sec.~4.1]{DH2}
and Cattani-Deligne-Kaplan's algebraicity theorem \cite[Thm~1.1]{CDK}, 
 one shows that Conjecture
\ref{into.conjhodge} is equivalent to Grothendieck's original formulation of the
variational Hodge conjecture and it would therefore
be a
consequence of the Hodge conjecture.

A $p$-adic analog of  Conjecture \ref{into.conjhodge} is suggested by Fontaine-Messing, it
is usually called the $p$-adic variational Hodge conjecture. Before we state it, we again
motivate it by the case of line bundles.

Let $k$ be a perfect field of characteristic $p>0$, $ W=W(k)$ be  the ring of Witt vectors
over $k$, $K={\rm frac}(W)$, $X/S$ be a smooth projective variety, $X_n\hookrightarrow X$ be the
closed immersion  defined by $(p^n)$; so $X_n=X\otimes_W W_n, W_n=W/(p^n)$.  Then
Berthelot constructs a crystalline-de Rham comparison isomorphism
\ga{}{ \Phi: H^i_{\dR}(X/W)\xrightarrow{\sim} H^i_{{\cris}} (X_1/W) , \notag }
which is recalled in Section \ref{criscoho}.
One also has a crystalline Chern character, see \eqref{cris.chernch},
\ga{}{\ch: K_0(X_1) \xrightarrow{} \bigoplus_r H^{2r}_{\cris}(X_1/W)_K. \notag}
Let us assume $p>2$. Then one has the exact obstruction sequence
\ga{intro.0}{ \varprojlim_n \Pic(X_n) \to \Pic(X_1)\xrightarrow{{\rm Ob}} H^2(X, p\sO_X) }
coming from the short exact sequence of sheaves
\ga{intro.1}{  1 \to (1+p\sO_{X_{n}})\to \sO^\times_{X_n}\to \sO^\times_{X_1}\to 1 }
and the $p$-adic logarithm isomorphism
\ga{intro.2}{   {\rm log}: 1+p\sO_{X_{n}}   \xrightarrow{\sim} p\sO_{X_n} . }
Grothendieck's formal existence theorem \cite[Thm.~5.1.4]{EGA3} gives an algebraization isomorphism
\ga{}{ \Pic(X)\xrightarrow{\sim} \varprojlim_n \Pic(X_n).\notag}
Using an idea of Deligne  \cite[p.~124 b)]{DelK3}, Berthelot-Ogus \cite{BO} relate  the
obstruction map in \eqref{intro.0}  to the Hodge level of the crystalline Chern class of a
line bundle. So altogether they prove  the line bundle version  of  Fontaine-Messing's $p$-adic variational  Hodge conjecture:

\begin{conj}\label{intro.padvarh}
 For $\xi_1\in K_0(X_1)_{\Q} $ such that 
$$\Phi^{-1}\circ \ch (\xi_1)\: \in\: \bigoplus_r F^rH^{2r}_{\dR}(X_K/K),$$ 
there is  a $\xi   \in K_0(X)_{\Q}$, such that $\ch(\xi|_{X_1})=\ch(\xi_1)\in 
\bigoplus_r H^{2r}_{\cris}(X_1/W)_K.
$ 
\end{conj}
In fact the conjecture can be stated more generally over any
$p$-adic complete discrete valuation ring with perfect residue field. 
Note that there is no analog of the absolute Hodge conjecture available over $p$-adic fields, which
would comprise the $p$-adic variational Hodge conjecture. So its origin is more mysterious
than the variational Hodge conjecture in characteristic zero.

Applications of Conjecture \ref{intro.padvarh} to modular forms are studied by Emerton and
Mazur, see \cite{Em}.

\medskip

We suggest to decompose the problem into two parts: firstly a formal deformation part and
secondly an algebraization part
\ga{}{
\xymatrix{  K_0(X) \ar[rr]  & &  \varprojlim_n K_0(X_n)  \ar@{~>>}@/^{0.8cm}/[ll]^{\rm
    algebraization} \ar[rr] & & K_0(X_1) \ar@{~>>}@/^{0.8cm}/[ll]^{\rm deformation} . } 
    \notag}

Unlike for $\Pic$, there is no general approach to the 
algebraization problem known. In this note, we study the deformation problem.  Our main
result, whose proof is finished in Section~\ref{cherniso}, states:

\begin{thm} \label{MainThm}
Let $k$ be a perfect field of characteristic $p>0$,  let $X/W$ be smooth projective scheme
over $W$ with closed fibre $X_1$. Assume  $p>d+6$, where $d=\dim(X_1)$. Then for $\xi_1\in
K_0(X_1)_{\Q}$ the following are equivalent
\begin{itemize} 
\item[(a)] we have
$$\Phi^{-1}\circ \ch (\xi_1)\:\in\: \bigoplus_{r} F^rH^{2r}_{\dR}(X/S),$$
\item[(b)]
there is  a $\hat{\xi}   \in \big( \varprojlim_n K_0(X_n)\big)_{\Q}$, such that $\hat{\xi}|_{X_1}=\xi_1 \in K_0(X_1)_{\Q}$.
\end{itemize}
\end{thm}

Before we describe the methods we use in our proof, we make three remarks.
\begin{itemize}
\item[(i)]
We do not handle the case where the ground ring is $p$-adic complete and ramified over
$W$. The reason is that we use techniques related to integral $p$-adic Hodge theory, which
do not exist over ramified bases. In fact, Theorem \ref{MainThm} is not integral, but a
major intermediate result, Theorem \ref{def.thm}, is valid with integral coefficients and
this theorem would not hold integrally over ramified bases.     
\item[(ii)]
 The precise form of the condition  $p>d+6$ on the characteristic has technical
  reasons. However, the rough condition that $p$ is big
  relative to $d$ is essential for our method for the same reasons explained in (i) for
  working over the base $W$.
\item[(iii)]
Note, we literally  lift the $K_0(X_1)_{\Q}$ class to an element in $\big(\varprojlim_n
K_0(X_n)\big)_{\Q}$, not only its Chern character in crystalline cohomology.  One thus
should expect that in order to algebraize $\hat{\xi}$ and in order to obtain the required class over
$X$ in Conjecture \ref{intro.padvarh}, one might have to move it to
another pro-class with the same Chern character. 
\end{itemize}

\smallskip

We now describe our method. 
We first construct for $p>r$ in an ad hoc way a motivic pro-complex $\Z_{X_\bull}(r)$ of
the $p$-adic formal scheme $X_\bull$ associated to $X$ on the Nisnevich site of $X_1$. 
For this we glue the Suslin-Voeveodsky motivic complex on $X_1$ with the Fontaine-Messing-Kato syntomic  complex  on $X_\bull$, see 
Definition \ref{Z(r)}.
In Sections \ref{ftriangle} and \ref{motiviccomplex} we  construct a fundamental triangle 
\ga{intro.fundtri}{    
 p(r)\Omega^{< r}_{X_\bull}[-1]  \to  \Z_{X_\bull}(r) 
\to 
 \Z_{X_1}(r) \to \cdots
 }
which in weight $r=1$ specializes to  \eqref{intro.1} and \eqref{intro.2}.
Here $p(r)\Omega^{< r}_{X_\bull}$ is a subcomplex of the truncated de Rham complex of
$X_\bull$, which is isomorphic to it tensor $\Q$. 

A. Beilinson translated back the existence of the fundamental triangle
\eqref{intro.fundtri} to give a definition of  $\Z_{X_\bull}(r)$ in the style of the Deligne
cohomology complex in complex geometry, which does not refer to the syntomic complex.  We
show in Appendix \ref{sec.cryscon} that there is a canonical isomorphism between his
definition and ours. Even if his definition is very elegant and it seems that one can
develop the theory completely along these lines,  we kept our viewpoint in the article. On one hand, syntomic cohomology  as developed in 
\cite{Ka2} and \cite{FonMes}  is well established, on the other hand, we need Kato's results on it to show our main theorem.


In Section \ref{hodgeobstruction}  we
define continuous Chow groups as continuous cohomology of our motivic
pro-complex by the Bloch type formula  $$\CH_{\cont}^r (X_\bull)=H^{2r}_{\cont}(X_1, \Z_{X_\bull}(r)).$$  
{From} \eqref{intro.fundtri} we obtain the higher codimension analog of the obstruction sequence \eqref{intro.0} 
  \begin{equation}\label{intro.obseq}
\CH^r_\cont(X_\bull) \to \CH^r(X_1)  \xrightarrow{\rm Ob}
H^{2r}_\cont(X_1,p(r)\Omega^{<r}_{X_\bull}) .
\end{equation}
 In Sections \ref{conmorph} and \ref{hodgeobstruction}  we relate the obstruction map in \eqref{intro.obseq} to the Hodge theoretic properties of the cycle class in crystalline cohomology.
Using this we prove the analog, Theorem \ref{def.thm}, of our Main Theorem \ref{MainThm} with $\varprojlim_n K_0(X_n)$ replaced by
$\CH_{\cont}(X_\bull)$.

We then define continuous $K$-theory  $K^{\cont}_0(X_\bull)$  of the $p$-adic formal
scheme $X_\bull$ in Section \ref{chern}. 
The continuous $K_0$-group maps surjectively to $\varprojlim_n K_0(X_n)$, so lifting classes
in $K_0(X_1)$ to continous $K_0$ is equivalent to lifting classes as in Theorem
\ref{MainThm}.

  Using the method of Grothendieck and Gillet \cite{Gil} and relying on
ideas of Deligne for the calculation of cohomology of classifying spaces, we define a Chern character 
\ga{intro.6}{\ch: K^{\cont}_0(X_\bull)_{\Q}\to \bigoplus_{r\le d}\CH^r_\cont(X_\bull)_{\Q}.}
Finally, using deep results from topological cyclic homology theory due to
Geisser-Hesselholt-Madsen, recalled in Section \ref{topcyc}, we
show in  Theorem~\ref{ciso.isothm} that the Chern character~\ref{intro.6} is an
isomorphism for $p>d+6$ by reducing it to an \'etale
local problem with $\Z/p$-coefficients. We also get a Chern character isomorphism on higher $K$-theory in Theorem~\ref{higherK}.
In Section~\ref{cherniso} we complete the proof of Theorem~\ref{MainThm}.
\\[.3cm]

\noindent

{\it Acknowledgements:} It is our pleasure to thank Lars Hesselholt for explaining to us
topological cyclic homology and Marc Levine for many important comments.  We also thank  Markus Spitzweck and Chuck Weibel for helpful discussions. We are
grateful to the mathematicians from the Feza G\"ursey Institute in Istanbul for giving us
the opportunity to present a preliminary version of our results in March 2011.  After our work was completed, Alexander Beilinson proposed to us an alternative definition of our motivic complex (see Section~\ref{sec.cand}). We thank him for his interest in our work and for his contribution to it. We also thank the various referees who sent comments to us.

\section{Crystalline and de Rham cohomology} \label{criscoho}

\noindent
In this section we study the de Rham complex of a $p$-adic formal scheme $X_\bull$ and the de
Rham-Witt complex of its special fibre $X_1$. We also introduce certain subcomplexes,
which coincide with the usual de Rham and de Rham-Witt complex tensor $\Q$. These
subcomplexes play an important r\^ole in the obstruction theory of cohomological Chow groups as
studied in Section~\ref{hodgeobstruction}.
We will think of the de Rham complex of $X_\bull$ and the de Rham-Witt complex of $X_1$ as pro-systems on the small Nisnevich site of $X_1$.

To fix notation let $S$ be a complete adic noetherian ring. Fix an ideal of definition
$\mathcal I\subset S$. We write $S_n=S/\mathcal I^n$.  Let $\Sch_{S_\bull}$ be the
category of $\mathcal I$-adic formal schemes $X_\bull$ which are
quasi-projective over ${\rm Specf}(S)$ and such that $X_n = X_\bull \otimes_S S/\mathcal
I^n $ is syntomic  \cite{FonMes}
over $S_n=S/\mathcal I^n$ for all $n\ge 1$.
By $\Sm_{S_\bull}$ we denote the full subcategory of $\Sch_{S_\bull}$ of formal schemes
which are (formally) smooth over $S_\bull$.

In the following  let $S=W=W(k)$ be the ring of Witt vectors of a
perfect field $k$, $p=\Char\, k>0$ and fix the ideal of definition $ \mathcal I=(p)$.
Let $X_\bull$ be in $\Sch_{W_\bull}$.

\begin{defn}
For $\bS_\et$ resp.\ $\bS_\Nis$ the small \'etale resp.\ Nisnevich site of $X_1$, we write
\begin{align*}
\S_\pro(X_1)_\etNis   \quad &\text{ for } \quad \S_\pro(\bS_\etNis)\\
\Sh_\pro(X_1)_\etNis \quad &\text{ for } \quad \Sh_\pro(\bS_\etNis)\\
\C_\pro(X_1)_\etNis \quad &\text{ for } \quad \C_\pro(\bS_\etNis)\\
\D_\pro(X_1)_\etNis \quad &\text{ for } \quad \D_\pro(\bS_\etNis),
\end{align*}
where the right hand side is defined in generality in Appendix~\ref{sec.hom}
and~\ref{sec.hoto}.
If we do not specify topology we usually mean Nisnevich topology.
\end{defn}

Note that the \'etale (resp.\  Nisnevich) sites of $X_1$ and $X_n$ ($n\ge 1$) are isomorphic.

\begin{defn} \label{sheaves} \mbox{}
\begin{itemize}
\item[(a)]
We define
\ga{}{ \Omega^\bulle_{ X_\bull} \in \C_\st (X_1)_\etNis    }
as the pro-system of de Rham complexes $n\mapsto \Omega^\bulle_{X_n / W_n}$.
\item[(b)]
We define 
\ga{}{ W_\bull \Omega^\bulle_{ X_1} \in \C_\st (X_1)_\etNis    }
as the pro-system of de Rham-Witt complexes \cite{Il}.
\end{itemize}

\end{defn}

\begin{defn}\label{cris.wlog}
 We define 
\ga{}{W_\bull \Omega^r_{X_1,\log} \in \Sh_\st (X_1)_\etNis   \notag}
as pro-system of \'etale or Nisnevich subsheaves in $W_n\Omega^r_{X_1}$ which are locally
generated by symbols
$$d\log \{ [a_1], \ldots ,[ a_r]\},$$ with $a_1 , \ldots , a_r\in
\sO_{X_1}^\times$  local sections and where $[ - ]$ is
the Teichm\"uller lift (\cite{Il}, p. 505, formula (1.1.7)).  
\end{defn} 

Clearly $\epsilon^*\, W_n\Omega^r_{X,\Nis} =
W_n\Omega^r_{X,\et}$ and Kato  shows \cite{Ka}
\begin{prop}\label{kato.thm}
 The natural map
\begin{equation}
W_n\Omega^r_{X,\log,\Nis} \xrightarrow{\sim}   \epsilon_*\, W_n\Omega^r_{X,\log,\et}
\end{equation}
is an isomorphism, in other words $ \epsilon_*\, W_n\Omega^r_{X,log,\et}$ is Nisnevich
locally generated by symbols in the sense of Definition~\ref{cris.wlog}.

\end{prop}

\medskip

\begin{defn}
For $r<p$ we define 
\ga{}{ p(r)\Omega^\bulle_{X_{\bull }}  \in \C_\pro (X_1)_\etNis
\notag }
as the de Rham complex
\ga{}{
p^{r}\sO_{X_\bull}
\to
p^{r-1}\Omega^1_{X_\bull}\to \ldots \to p\Omega^{r-1}_{
X_\bull}\to \Omega^r_{ X_\bull}\to \Omega^{r+1}_{X_\bull}\to \ldots
.\notag}

\medskip

\medskip

For $r<p$ we define 
\ga{}{ q(r) W_\bull \Omega^\bulle_{X_1}\in   \C_\pro (X_1)_\etNis  \notag}
as the de Rham--Witt complex
\ga{}{ 
p^{r-1}VW_\bull \sO_{X_1} \to
p^{r-2} VW_\bull \Omega^1_{X_1}  \to \ldots \notag \\ \to pVW_\bull\Omega^{r-2}_{X_1}\to
VW_\bull\Omega^{r-1}_{X_1}\to 
 W_\bull\Omega^{r}_{X_1}\to 
W_\bull\Omega^{r+1}_{X_1}\to \ldots  \notag}
here $V$ stands for the  Verschiebung homomorphism  (see
\cite[p.~505]{Il}).
\end{defn}

\medskip

\begin{rmk}
It is of course possible to define analogous complexes  $p(r) \Omega^\bulle_{ X_\bull}$ and  $ q(r)
W_\bull\Omega^\bulle_{X_1}$ in case $r\ge p$ by introducing divided powers \cite{FonMes}. Unfortunately, doing so introduces a number of problems both with regard to syntomic cohomology and later in Section~\ref{topcyc}, so we have chosen to assume $r<p$ throughout. 
\end{rmk}

\medskip

In the rest of this section we explain the construction of
canonical isomorphisms
\begin{align}
\label{2.4a} \Omega^\bulle_{X_\bull}  \simeq W_\bull \Omega_{X_1}^\bulle &\quad
\text{ in } \quad \D_\pro (X_1 ) \\ 
\label{2.5a} p(r) \Omega^\bulle_{X_\bull}  \simeq q(r) W_\bull \Omega_{X_1}^\bulle &\quad
\text{ in } \quad \D_\pro (X_1 ) .
\end{align}

\medskip

Recall the following construction, see \cite[Sec.~II.1]{Il}, \cite[Section~1]{Ka2}.
For the moment we let $X_\bull$ be a not necessarily smooth object in $\Sch_{W_\bull}$. 
We fix a closed embedding $ X_\bull \to Z_\bull$, where $Z_\bull/W_\bull$ in
$\Sm_{W_\bull}$ is endowed with a lifting
$F:Z_\bull\to Z_\bull$ over
$F: W_\bull \to W_\bull $ of Frobenius on $Z_1$. One defines the PD envelop $ X_n \to
D_n=D_{X_n}(Z_n)$. Recall that $D_n$ is endowed with a de Rham
complex $\Omega^\bulle_{D_n/W_n}:=\sO_{D_n}\otimes_{\sO_{Z_n}}
\Omega^\bulle_{Z_n/W_n}$ satisfying
$d\gamma^n(x)=\gamma^{n-1}(x)\, dx$ where $n!\cdot\gamma^n(x)=x^n.$
We define $J_n$ to be the ideal of $X_n\subset D_n$ and $I_n=(J_n,
p)$ to be the ideal sheaf of $X_1\subset D_n$. Then $J_n$ and $I_n$
are nilpotent sheaves on $X_{1,\et}$ with divided powers
$J_n^{[j]}$ and $I_n^{[j]}$. If $j<p$ one has $J_n^{[j]}=J_n^j$ and $  I_n^{[j]}=I_n^j$.

As before the \'etale (resp.\ Nisnevich) sites of $X_1$ and $D_n$ ($n\ge 1$) are isomorphic. 
In the following by abuse of notation we identify these equivalent sites.

\medskip

We continue to assume $r<p$. 
\begin{defn}(see \cite[p.211]{Ka2})
 One defines $J(r)\Omega^\bulle_{D_{\bull}} \in \C_\pro(D_\bull)_\etNis$
as the complex
\[
J^r_\bull \to  
J^{(r-1)}_\bull  \otimes_{\sO_{Z_\bull}} \Omega^1_{Z_{\bull}} 
\to
\ldots \to J_\bull \otimes_{\sO_{Z_\bull}}   \Omega^{r-1}_{Z_{\bull}}\to 
\sO_{D_\bull} \otimes _{\sO_{Z_\bull}} \Omega^{r}_{Z_{\bull}}  \to
\ldots.
\]
One defines $I(r)\Omega^\bulle_{D_\bull} \in \C_\pro( D_\bull )_\etNis$
as the complex
\[
I^r_\bull\to  
I^{(r-1)}_\bull   \otimes_{\sO_{Z_\bull}} \Omega^1_{Z_{\bull}} 
\to
\ldots \to I_\bull \otimes_{\sO_{Z_\bull}}   \Omega^{r-1}_{Z_{\bull}}\to 
\sO_{D_\bull}\otimes_{\sO_{Z_\bull}} \Omega^{r}_{Z_{\bull}}  \to
\ldots .
\]

\end{defn}

\medskip
For the rest of this section we assume $X_\bull$ is in $\Sm_{W_\bull}$.
The lifting of Frobenius $F$ defines a morphism \ga{}{ \sO_{D_n}\to
  \prod_1^{n}\sO_{D_n}, \ x\mapsto (x, F(x), \ldots,
  F^{n-1}(x)),\notag} which induces a well defined morphism $\Phi(F):
\sO_{D_n}\to W_n\sO_{X_1}$, which in turn induces a quasi-isomorphism
of differential graded algebras  \cite[Sec.~II.1]{Il}
\begin{equation}
\Phi(F): \Omega^\bulle_{D_n}\to
W_n\Omega^\bulle_{X_1} .
\end{equation}
The restriction
homomomorphisms
\begin{align}
\Omega^\bulle_{D_n}  &\xrightarrow{\sim} \Omega^\bulle_{X_n}\label{2.7a}  \\
J(r) \Omega^\bulle_{D_n}  &\xrightarrow{\sim} \Omega^{\ge r}_{X_n} \label{2.8a} \\
\label{cris.iso} I(r) \Omega^\bulle_{D_n}  &\xrightarrow{\sim} p(r) \Omega^\bulle_{X_n}
\end{align}
 are quasi-isomorphisms of differential graded algebras \cite[7.26.3]{BO}.
We get isomorphisms
\ga{cris.diaco}{
\xymatrix{\Omega^\bulle_{X_\bull} \ar@{..}[rd]_{(*)}  & \ar[l]_-{\sim} \ar[d]^-{\Phi(F)}_-{\wr}
 \Omega^\bulle_{D_\bull} \\
& W_\bull\Omega^\bulle_{X_1}
}}
which induce a canonical dotted isomorphism $(*)$ in $\D_\pro(X_1)_\etNis$, independent of the choice of $Z$.

\begin{prop} \label{qis}
For $X_\bull \in \Sm_{W_\bull}$
the diagram \eqref{cris.diaco} induces the diagram
\ga{}{\xymatrix{p(r)\Omega^\bulle_{X_\bull}  \ar@{..}[rd]_{(*)}  & \ar[l]_-{\sim}
\ar[d]^-{\Phi(F)}_-{\wr}
I(r) \Omega^\bulle_{D_\bull} \\
& q(r)W_\bull \Omega^\bulle_{X_1}
 \notag
}
}
whose maps are isomorphisms in $\D_\pro(X_1)_\etNis$. They induce a canonical isomorphism $(*)$,
independent of the choice of $Z$.
\end{prop}

\begin{proof}
We have to show that $\Phi(F)$ is an isomorphism in $\D_\pro(X_1)_\etNis$.
By \eqref{cris.iso} we can without loss of generality   assume $X_\bull=Z_\bull=D_\bull$
are affine  with Frobenius lift $F$.
Let $d= \dim\, X_1 $. Consider sequences $\nu_*:=\nu_0\ge \nu_1\ge\cdots\ge
\nu_d \ge \nu_{d+1}\ge 0$
with $\nu_{i+1}\ge \nu_i-1$ and $\nu_i<p$ for all $0\le i\le d$. We also assume
$\nu_{d+1}=\max (0,\nu_d-1)$. To any such sequence we associate a subcomplex
$q(\nu_*)W_\bull \Omega^\bulle_{X_1}$ of $W_\bull\Omega^\bulle_{X_1}$  as follows:
\eq{4.5}{q(\nu_*)W_\bull\Omega^i_{X_1} = \begin{cases} p^{\nu_i}W_\bull\Omega^i_{X_1} & \text{ for }\;\;
\nu_i=\nu_{i+1}
\\
p^{\nu_{i+1}}VW_\bull\Omega^i_{X_1} & \text{ for }\;\; \nu_i = \nu_{i+1}+1
\end{cases}
}
This is indeed a subcomplex (because $VW_\bull\Omega^i_{X_1} \supset
pW_\bull \Omega^i_{X_1}$). 
correspond to the
sequence $\nu_i=\max (0,r-i)$.  We get a map
\eq{4.6}{\Phi(F):   p^{\nu_\bulle}\Omega^\bulle_{X_\bull}  \to q(\nu_*)W_\bull
  \Omega^\bulle_{X_1} .
}

\begin{lem} 
The map $\Phi(F)$ in \eqref{4.6} induces an isomorphism in $\D_\pro(X_1)_\etNis$.  
\end{lem}
We proceed by induction on $N=\sum \nu_i$. If $N=0$ the assertion
is that $\Omega^\bulle_{A_\bull} \to W\Omega_{A_1}$ is a quasi-isomorphism, which
is Illusie's
result \cite[Thm.~II.1.4]{Il}. Suppose $N>0$ and
assume the result for smaller values of $N$. Let
$i$ be such that $\nu_0=\cdots = \nu_i >\nu_{i+1}$. Define a sequence $\mu_*$
such that $\mu_j=\nu_j$ for $j\ge i+1$ and such that $\mu_j=\nu_j-1$ for $j\le i$.  By induction
$p^{\mu_\bulle}\Omega^\bulle_{X_\bull} \to q(\mu_*)W_\bull\Omega^\bulle_{X_1}$ is an
isomorphism in $\D_\pro(X_1)_\etNis$. One has, up to
isomorphism 
\ga{4.8}{p^{\mu_\bulle}\Omega^\bulle_{X_\bull}/p^{\nu_\bulle}\Omega^\bulle_{X_\bull} \cong
\mathcal{O}_{X_1}\to \cdots \to
\Omega^i_{X_1} \\
\label{4.9} q(\mu_*)W_\bull \Omega^\bulle_{X_1}/q(\nu_*)W_\bull \Omega^\bulle_{X_1}
\cong \\
W(X_1)/pW(X_1) \to \cdots \to W_\bull\Omega^{i-1}_{X_1}/pW_\bull \Omega^{i-1}_{X_1}\to 
W_\bull\Omega^{i}_{X_1}/VW_\bull\Omega^{i}_{X_1}
\notag
}
Complexes \eqref{4.8} and \eqref{4.9} are quasi-isomorphic by
\cite[Cor.~I.3.20]{Il},
 proving the lemma. Note we are using throughout that
multiplication by $p$ is a monomorphism on  $W_\bull\Omega^\bulle_{X_1}$. 
\end{proof}

\medskip

For $X_1 /k$ projective we work with the {\it crystalline cohomology} groups
\begin{equation}
H_{\cris}^i(X_1/W) = H^i_\cont( X_1, W_\bull
\Omega^\bulle_{X_1} )
\end{equation}
 and the {\it refined crystalline cohomology} groups 
 $H^i_\cont( X_1,q(r) W_\bull \Omega^\bulle_{X_1} )$.
The definition of continuous cohomology groups is recalled in Definition~\ref{hoto.contcoho}.
 Note that because $H^i( X_1,
 W_n \Omega^r_{X_1} )$ are $W_n$-modules of finite type, we have 
\begin{align*}
H^i_\cont( X_1, W_\bull
\Omega^\bulle_{X_1} ) &= \varprojlim_n H^i( X_1, W_n
\Omega^\bulle_{X_1} )\\
H^i_\cont( X_1, q(r) W_\bull
\Omega^\bulle_{X_1} ) &= \varprojlim_n H^i( X_1,q(r) W_n
\Omega^\bulle_{X_1} ).
\end{align*}
For the same reason we have for de Rham cohomology
\begin{align*}
H^i_\cont( X_1, 
\Omega^\bulle_{X_\bull} ) &= \varprojlim_n H^i( X_1, 
\Omega^\bulle_{X_n} )\\
H^i_\cont( X_1, 
p(r)\Omega^\bulle_{X_\bull} ) &= \varprojlim_n H^i( X_1, 
p(r)\Omega^\bulle_{X_n} ).
\end{align*}
In particular if $X_\bull $ is the $p$-adic formal scheme associated to a smooth
projective scheme $X/W$ we get $H^i_\cont( X_1, 
\Omega^\bulle_{X_\bull} ) = H^i(X, \Omega_{X/W}^\bulle) $ by \cite[Sec.\ 4.1]{EGA3}.

\medskip

Gros \cite{G} constructs the crystalline Chern character
\begin{equation}\label{cris.chernch}
K_0(X_1)  \xrightarrow{\ch} \bigoplus_{r} H^{2r}_\cris (X_1/W)_\Q
\end{equation}
using the method of Grothendieck, i.e.\ using the projective bundle formula.  
The crystalline Chern character is a ring  homomorphism.

\section{A Candidate for the Motivic Complex, after A. Beilinson} \label{sec.cand}

\noindent We continue to assume $X_\bull$ is a smooth, projective formal scheme over $S=\text{Spf}(W(k))$, and we write $X_n=X_\bull \times_S \Spec(W/p^nW)$. In particular, $X_1$ is the closed fibre. The main goal of this paper is to relate the continuous $K_0(X_\bull)$ to the cohomology (in the Nisnevich topology) of suitable motivic complexes $\Z_{X_\bull}(r)$.
In this section we introduce briefly the referee's candidate for  $\Z_{X_\bull}(r)$. We work in the Nisnevich topology on $X_1$. Let $\Z_{X_1}(r)$ be the motivic complex in the Nisnevich topology on $X_1$. (For details see section \ref{motiviccomplex}.) The motivic complex on $X_1$ is linked to crystalline cohomology by a $d\log$ map  (compare \eqref{7.4a} and definition \ref{cris.wlog})
\eq{3.1a}{\Z_{X_1}(r) \to \sK^M_{X_1,r}[-r] \to W_\bull\Omega^r_{X_1,\log}[-r] \inj q(r)W_\bull\Omega^\bulle_{X_1}
}
The Chow group $CH^r(X_1) \cong H^{2r}(X_1, \Z_{X_1}(r))$ and the crystalline cycle class
$CH^r(X_1) \to H^{2r}(X_1, q(r)W_\bull \Omega^\bulle_{X_1})\to H^{2r}_{crys}(X_1/W) $ is the map on cohomology from \eqref{3.1a}. On the other hand, $H^{2r}(X_1, q(r)W_\bull\Omega^\bulle_{X_1})\cong H^{2r}(X_\bull, p(r)\Omega^\bulle_{X_\bull})$ (proposition \ref{qis}), and the Hodge obstruction to the cycle on $X_1$ lifting to $X_\bull$ is the composition
$$CH^r(X_1) \to H^{2r}(X_\bull, p(r)\Omega^\bulle_{X_\bull}) \to H^{2r}(X_\bull, p(r)\Omega^{\le r-1}_{X_\bull}).
$$
Thus, to measure the Hodge obstruction, it is natural to look for the cohomology of some sort of cone 
\eq{3.2a}{H^*(X_1,\Z_{X_1}(r) \xrightarrow{?} p(r)\Omega^{\le r-1}_{X_\bull})
}
analogous to the cone $H^*(X, \Z_X(r) \to \Omega^{\le r-1}_X)$ defining Deligne cohomology in characteristic $0$. Unfortunately, the arrow $?$ in \eqref{3.2a} is only defined in the derived category, so the cohomology is only given up to non-canonical isomorphism. To remedy this, we consider a more elaborate cone. We choose a divided power envelope $X_\bull \inj D_\bull$ as in section \ref{criscoho}.  Let $I_\bull \subset \sO_{D_\bull}$ be the ideal of $X_1$ and consider the cone
\ml{3.3a}{\widetilde\Z_{X_\bull}(r):= {\rm Cone}\Big(I(r)\Omega^\bulle_{D_\bull}\oplus \Omega^{\ge r}_{X_\bull}\oplus \Z_{X_1}(r) \xrightarrow{\phi}\\  p(r)\Omega^\bulle_{X_\bull}\oplus q(r)W\Omega^\bulle_{X_1}\Big)[-1].
}
Thinking of $\phi = (\phi_{ij})$ as a $2$ by $3$ matrix operating on the left on the domain viewed as a column vector with $3$ entries, we have $\phi_{1,1}: I(r)\Omega^\bulle_{D_\bull} \to p(r)\Omega^\bulle_{X_\bull}$ and $\phi_{2,1}:I(r)\Omega^\bulle_{D_\bull} \to q(r)W\Omega^\bulle_{X_1}$ defined as in proposition \ref{qis}. The map $\phi_{1,2}:\Omega^{\ge r}_{X_\bull}\inj p(r)\Omega^\bulle_{X_\bull}$ is the natural inclusion, and $\phi_{2,3}: \Z_{X_1}(r) \to q(r)W\Omega^\bulle_{X_1}$ is \eqref{3.1a}. The other entries of $\phi$ are zero. We will show in appendix \ref{sec.cryscon} that $\widetilde\Z_{X_\bull}(r) \simeq \Z_{X_\bull}(r)$ where the complex on the right is given in definition \ref{Z(r)}. By crystalline theory, a different choice $E_\bull$ of divided power envelope yields a canonical quasi-isomorphism $\Omega^\bulle_{D_\bull}\simeq \Omega^\bulle_{E_\bull}$ in the derived category, and hence the object $\widetilde\Z_{X_\bull}(r)$ is canonically defined in the derived category.

\section{Syntomic complex and de Rham-Witt sheaves}\label{sec.synt}

\noindent
We introduce the syntomic complex \cite{Ka2} in the \'etale and Nisnevich topologies and collect some facts
about de Rham-Witt sheaves.

Let $X_\bull$ be in $\Sch_{W_\bull}$ and let $X_\bull \hookrightarrow D_\bull$ be as in Section~\ref{criscoho}.
Assume $r<p$. 
Then the morphism $\Omega^\bulle_{D_n}
\xrightarrow{p^r} \Omega^\bulle_{D_{n+r}}$ of complexes of sheaves on 
$X_{1,\et}$ is injective, and the Frobenius map
\[
 J(r)\Omega^\bulle_{D_{n+r}}  \xrightarrow{F}
\Omega^\bulle_{D_{n+r}} 
\]
 factors through 
$\Omega^\bulle_{D_n}
\xrightarrow{p^r} \Omega^\bulle_{D_{n+r}}$, see  \cite[Section~1]{Ka2}.

\medskip

\begin{defn} \label{frJ} (\cite[Cor.1.5]{Ka2}) 
One defines the morphism  $$f_r: J(r)\Omega^\bulle_{D_\bull} \to
\Omega^\bulle_{D_\bull}$$
of complexes in $\Sh_\pro(X_{1})_\et$
via the 
factorization
\ga{}{ F:   J(r)\Omega^\bulle_{D_{n+r}}   
\to 
J(r)\Omega^\bulle_{D_n}    \xrightarrow{f_r}  
\Omega^\bulle_{D_n}  \xrightarrow{p^r}
\Omega^\bulle_{D_{n+r}}
\notag}
of the Frobenius $F$. 

\end{defn}

Note that $f_r$ is defined using the existence of $X_{n+r}$, not directly on
$X_n$.

\begin{defn}(\cite[Defn.~1.6]{Ka2}) \label{S}
We define the {\it syntomic complex} $\frak{S}_{X_{ \bull}} (r)_\et $ in the \'etale topology 
by 
\ga{}{ \frak{S}_{X_{\bull}}(r)_\et = \cone \big(
  J(r)\Omega^\bulle_{D_\bull}\xrightarrow{1-f_r}
  \Omega^\bulle_{D_\bull} \big) [-1] , \notag } 
which we usually
consider as an object in $\D_\pro(X_1 )_\et$.

In the Nisnevich topology
we define $\frak{S}_{X_{\bull}}(r) \in \D_\pro(X_1)_\Nis$ to be 
 \ga{}{
 \frak{S}_{X_\bull}(r) =\tau_{\le r} R\epsilon_* \frak{S}_{X_{\bull}}(r)_\et \notag .}
\end{defn}

\smallskip

Here  $\epsilon: X_{1,\et} \to X_{1,\Nis}$ is the morphism of sites and $\tau_{\le
  r}$ is the `good' truncation.
This definition does not depend on the choices $(Z,F)$, see comment after
\cite[Defn.~1.6]{Ka2}.

It is well known, see \cite[Thm.~3.6(1)]{Ka2}, that \[\epsilon^* \, \frak{S}_{X_{\bull}}(r) = \frak{S}_{X_{\bull}}(r)_\et.\]

For the rest of this section let $X_1$ be a smooth quasi-projective scheme over $k$ and
let $p,r\in \N$ be arbitrary.
Recall from \cite[Prop.~I.3.3,~(3.3.1)]{Il} that the internal Frobenius 
$W_{n+1}\Omega^r_{X_1} \xrightarrow{F} W_n\Omega^r_{X_1}$ induces a well defined homomorphism
$$F_r: W_n\Omega^r_{X_1}\to W_n\Omega^r_{X_1}/dV^{n-1}\Omega^{r-1}_{X_1}$$
by first lifting local sections of $W_n\Omega^r_{X_1}$ to
$W_{n+1}\Omega^r_{X_1}$ and then applying $F$ to it. 
Furthermore, by definition of $f_r$, one has a commutative diagram in $\Sh_\pro(X_1)$
\ga{}{\xymatrix{\ar[d]_{\Phi(F)}
J(r)\Omega^r_{D_{\bull}}   \ar[r]^-{f_r}&
\Omega^r_{D_{\bull}} \ar[d]^{\Phi(F)}\\
 W_\bull\Omega^r_{X_1}\ar[r]^-{F_r}&
W_\bull\Omega^r_{X_1}/dV^{n-1}\Omega^r_{X_1}
}\notag
}

\begin{lem} \label{Wlog}
 One has  a short exact
sequence  
\ga{}{0\to W_n\Omega^r_{X_{1,\log}}\to
W_n\Omega^r_{X_1}/dVW_{n-1}\Omega^{r-1}_{X_1}
\xrightarrow{1-F_r} 
W_n\Omega^r_{X_1}/dW_n \Omega^{r-1}_{X_1} \to 0  \notag}
on  $X_{1,\et}$. On $X_{1,\Nis}$ the sequence is still exact on the left and in the middle.
\end{lem}
\begin{proof}
Consider first the situation in the \'etale topology. One has a commutative diagram with exact columns 
\begin{small}
\[
\xymatrix{0 \ar[r] &
 W_n\Omega^r_{X_{1,\log}}  \ar[r]  &
W_n\Omega^r_{X_1}/dVW_{n-1}\Omega^{r-1}_{X_1}
\ar[r]^-{ 1-F_r}  &
W_n\Omega^r_{X_1}/dW_n \Omega^{r-1}_{X_1} \ar[r] & 0 \\
0 \ar[r]  &
 W_n\Omega^r_{X_{1,\log}}  \ar[r] \ar@{=}[u] &
W_n\Omega^r_{X_1} 
\ar[r]^-{1-F_r}  \ar@{->>}[u]   &
W_n\Omega^r_{X_1}/dV^{n-1}\Omega^{r-1}_{X_1}\ar@{->>}[u] \ar[r] & 0  \\
& & dVW_{n-1}\Omega^{r-1}_{X_1} \ar[r]^-{\phi:=1-F_r}  \ar@{>->}[u] &  
dW_n \Omega^{r-1}_{X_1} / dV^{n-1}\Omega^{r-1}_{X_1}  \ar@{>->}[u] &
}
\]
\end{small}

By \cite[Lem.~1.2]{CTSS} the middle row is exact. 
Thus the top row is exact if and only if the map $\phi$ is an isomorphism.

The map $V : dW_n\Omega^{r-1}_{X_1} \to W_{n+1}\Omega^{r}_{X_1}$ is
divisible
by $p$. Denote by $\psi$ the factorization 
\[
V: dW_n\Omega^{r-1}_{X_1}  \xrightarrow{\psi}
 W_n\Omega^{r}_{X_1}  \xrightarrow{p}  W_{n+1}\Omega^{r}_{X_1}.
\]
The image of $\psi$  lies in $dVW_{n-1}\Omega^{r-1}_{X_1}$ as $Vd=pdV $.
The inverse of $\phi$ is given by $\psi + \psi^2 + \psi^3 + \cdots$.

Finally, for the Nisnevich topology, starting with the basic result for a coherent sheave $E$ that $\epsilon_*E_{\et} = E_{\Nis}$ and $R^i\epsilon_*E_{\et}=(0)$ for $i\ge 1$, one gets $\epsilon_*W_n\Omega^r_{X_1,\et} = W_n\Omega^r_{X_1,\Nis}$. Then, using results from \cite{Il}, Section 3.E, p.~579, one gets 
\[\epsilon_*\Big( W_n\Omega^r_{X_1,\et}/dVW_{n-1}\Omega^{r-1}_{X_1,\et}\Big) =
W_n\Omega^r_{X_1,\Nis}/dVW_{n-1}\Omega^{r-1}_{X_1,\Nis}.
\]
One concludes using proposition \ref{kato.thm} and left-exactness of $\epsilon_*$. 
\end{proof}

\smallskip

Denote by $F_r:\tau_{\ge r} q(r) W_n \Omega^\bulle_{X_1} \to \tau_{\ge r}  W_n
\Omega^\bulle_{X_1} $ the morphism which in degree $r+i$ is induced by $p^i\, F$.

\begin{lem}\label{syn.lem2}
For $i>0, r\ge 0$ the map
\[
(1-F_r ):  W_n \Omega^{r+i}_{X_1} \to W_n \Omega^{r+i}_{X_1}
\]
is an isomorphism in $\Sh(X_1)_\etNis$.
\end{lem}
\begin{proof} This is \cite[I.Lem.3.30]{Il}.
\end{proof}

\smallskip

In $\Sh_\pro (X_1)_{\Nis}$  the internal Frobenius 
$F:q(r) W_\bull \Omega^i_{X_1} \to W_\bull \Omega^i_{X_1} $ is divisible by $p^{r-i}$ for
$i<r$. Indeed, for a local section $p^{r-1-i}V\alpha\in q(r)W_\bull\Omega^i_{X_1}$,  $F(p^{r-1-i}V\alpha)= p^{r-1-i} FV(\alpha) $ and $FV=p$ (\cite[I.~Lem.4.4] {Il}).
We denote this divided Frobenius by $$F_r: q(r) W_\bull \Omega^i_{X_1} \to W_\bull
\Omega^i_{X_1} $$ as a morphism in $\C_\pro(X_1)_\Nis$. 

\smallskip

\begin{lem}\label{syn.lem3}
In $\D_\pro (X_1)_\etNis$ the map
\[
(1-F_r ): \tau_{<r} q(r) W_\bull \Omega^\bulle_{X_1} \to \tau_{<r}  W_\bull \Omega^\bulle_{X_1}
\]
becomes an isomorphism.
\end{lem}

\begin{proof}
Applying \cite[I,~Lem.~4.4]{Il}, 
one has for $i\le r-1$ and $\alpha$ a local section in $  W_\bull\Omega^i_{X_{1,
\rm \acute{e}t }}$
\ga{}{(1-F_r)(-p^{r-i-1}V\alpha)= \alpha- p^{r-i-1} V\alpha,           
\notag}
thus
\ga{}{ \alpha= (1-F_r)(\beta), \ \beta=-(p^{r-1-i}V)\sum_{n=0}^\infty
(p^{r-1-i}V)^n(\alpha) \in p^{r-i-1}VW_{\bull}\Omega^i_{X_{1, \rm{\acute{e}t   } }}.
 \notag }
On the other hand,  clearly if $ W_\bull\Omega^i_{X_{1,
\rm \acute{e}t }} \ni \alpha=  p^{r-i-1} V\alpha$, then $\alpha \in  (p^{r-i-1}
V)^n W_\bull\Omega^i_{X_{1,
\rm \acute{e}t }}$ for all $n\ge 1$, thus $\alpha=0$.  This finishes the proof.

\end{proof}

\smallskip

Putting Lemmas \ref{Wlog}, \ref{syn.lem2} and \ref{syn.lem3} together we get

\begin{cor}\label{syn.cor1}
In $\D_\pro(X_1)_\et$ there is an exact triangle
\[
W_\bull \Omega^r_{X_1,\log} [-r] \to q(r) W_\bull \Omega^\bulle_{X_1}
\xrightarrow{1-F_r}  W_\bull \Omega^\bulle_{X_1} \xrightarrow{[1]} \cdots  .
\]
\end{cor}

\bigskip

\begin{rmk}\label{synt.rmkbig}
To end this section we remark that one can define the syntomic complex in
$D_\pro(\Sch_{W_\bull,\etNis})$, where $\Sch_{W_\bull,\etNis}$ is the big \'etale
resp.\ Nisnevich site with underlying category $\Sch_{W_\bull}$. For this one uses the
syntomic site and the crystalline Frobenius
instead of the immersion $X_\bull \hookrightarrow Z_\bull$ and the Frobenius lift on $Z_\bull$, see \cite{GK}, \cite{FonMes}.
\end{rmk}

\section{Fundamental triangle}\label{ftriangle}

\noindent
Let $X_\bull $ be in $\Sm_{W_\bull}$ and assume $r<p$.
The goal of this section is to decompose the Nisnevich syntomic complex
$\frak{S}_{X_\bull}(r)$ in a part $W_\bull \Omega^r_{X_1,\log}[-r]$
stemming from the reduced fibre $X_1$ and a `deformation part' $p(r) \Omega^{<r}_{X_\bull}[-1]$.

As a technical device we need a variant of the syntomic complex with $J(r)$ replaced by
$I(r)$. In analogy with  Definition~\ref{frJ} we propose:

\begin{defn}  \label{frI}
Let $f_r$ be  the canonical  factorization of Frobenius map 
$$F: I(r)\Omega^\bulle_{D_{n+r}} \xrightarrow{f_r} \Omega^\bulle_{D_n} \xrightarrow{
   p^r} \Omega^\bulle_{D_{n+r}}.$$
\end{defn} 

Note that  this time there is no factorization of the form
$$f_r :  I(r)\Omega^\bulle_{D_{n+r}}  \xrightarrow{\rm rest}
I(r)\Omega^\bulle_{D_n}  \to \Omega^\bulle_{D_n}.$$
We write
\[
 I(r)\Omega^\bulle_{D_\bull} \xrightarrow{f_r} \Omega^\bulle_{D_\bull}
\] 
for the induced  morphism  in $\C_\pro(X_1)$.

\begin{defn}

One defines 
$$ \frak{S}^I_{X_\bull}(r)_\et =\cone  (     I(r)\Omega^\bulle_{D_\bull}
\xrightarrow{1-f_r}   \Omega^\bulle_{D_\bull}) [-1]  $$
in $\D_\pro(X_1)_\et$. 
In the  Nisnevich topology we define
\[
 \frak{S}^I_{X_{\bull}}(r) = \tau_{\le r}  R\epsilon_* \, \frak{S}^I_{X_{\bull}}(r)_\et 
\] 
in $\D_\pro(X_1)_\Nis$.
\end{defn}

\smallskip

\begin{prop} \label{SI}
For $X_\bull$ in $\Sm_{W_\bull}$ the map $\Phi(F)$
induces an isomorphism 
\ga{}{\frak{S}^I_{X_\bull}(r)_\et \xrightarrow{\Phi^I}
W_\bull \Omega^r_{X_1,\log}[-r]   \notag     }
in $\D_\pro(X_1)_\et$. 
In particular applying the composed functor $\tau_{\le r}\circ R\epsilon_*$   we also get an
isomorphism 
\[
\frak{S}^I_{X_\bull}(r) \xrightarrow{\Phi^I}
W_\bull \Omega^r_{X_1,\log}[-r]
\] 
in $\D_\pro (X_1)_\Nis$.
\end{prop}

\begin{proof}
Indeed we have the chain of isomorphisms in $\D_\pro(X_1)_\et$.
\begin{equation}\label{syn.chainofmaps2} \xymatrix{
\frak{S}_{X_\bull}^I (r)_\et \ar[d]^{\Phi(F)}_-{(1)} \\
\cone ( q(r) W_\bull \Omega^\bulle \xrightarrow{1-F_r}  W_\bull
\Omega^\bulle   ) [-1] \ar[d]_-{(2)}   \\
  \cone ( W_\bull \Omega^r/ dVW_{\bull}\Omega^{r-1}  \xrightarrow{1-F_r}    W_\bull \Omega^r /
  dW_\bull\Omega^{r-1}   )[-r-1]      \\
W_\bull \Omega^r_{X_1,\log} [-r]  \ar[u]^-{(3)}
}
\end{equation}
where $\rm (1)$ is an isomorphism by Proposition~\ref{qis}, $\rm (2)$ is defined and  an isomorphism by
Lemmas \ref{syn.lem2} and \ref{syn.lem3} and $\rm (3)$ is an isomorphism by Lemma
\ref{Wlog}.

For Nisnevich topology we have $$  \tau_{\le 0}\circ R \epsilon_*\, W_n\Omega^r_{X,\log,\et}=  \epsilon_*\, W_n\Omega^r_{X,\log,\et}= W_n\Omega^r_{X,\log,\Nis} $$
by Proposition~\ref{kato.thm}.
\end{proof}

Recall that we work in Nisnevich
topology if not specified otherwise.

\begin{thm}[Fundamental triangle] \label{syntF}
For $X_\bull $ in $\Sm_{W_\bull}$ one has an exact triangle
\ga{}{      
p(r)\Omega^{< r}_{X_\bull}[-1] \to 
\frak{S}_{ X_\bull}(r) \xrightarrow{\Phi^J} W_\bull \Omega^r_{X_1,
\log}[-r]
\xrightarrow{[1]}       \ldots                 \notag}
in  $\D_\pro(X_1)$.
In particular, the support of $\frak{S}_{ X_\bull }(r) $ lies in degrees $[1 ,  
r]$ for $r\ge 1$.

\end{thm}

\begin{proof}
 We first construct the \'etale version of the triangle.
Let
\[
\mathfrak W(r)= \cone  ( J(r)\Omega^\bulle_{D_{\bull}} \xrightarrow{}
I(r)\Omega^\bulle_{D_{\bull}})[-1] .
\]

 Proposition~\ref{SI} implies that one
has an exact triangle 
\ga{fund}{ \mathfrak W(r)  \to 
\frak{S}_{ X_\bull}(r)_\et \xrightarrow{\Phi^J} W_\bull \Omega^r_{X_1,
\log}[-r]
\xrightarrow{[1]}       \ldots                 }
in $\D_\pro(X_1)_\et$.

By Proposition~\ref{qis} we conclude that the restriction
map from $D_\bull$ to $X_\bull$ induces an isomorphism 
\ga{}{ \mathfrak W(r)   \xrightarrow{{\rm rest} } 
p(r)\Omega^{\le r-1}_{
X_\bull }[-1]
  \notag}
in $\D_\st(X_1)_\et$. 

We now come to the Nisnevich version. One has to show that applying $\tau_{\le r}\circ R \epsilon_*$
to  exact triangle \eqref{fund}, one obtains an exact triangle in Nisnevich 
topology.
  One has an isomorphism  
\[
\epsilon_* p(r)
\Omega^{\le r-1}_{
X_\bull}[-1]\xrightarrow{\simeq}  R\epsilon_* p(r)\Omega^{\le r-1}_{
X_\bull}[-1]
\]
in $\D_\st(X_1)_\Nis$,
 thus in particular the latter complex  has support in cohomological degrees  $[1, r]$. 
Applying Lemma \ref{homalg.lem1} finishes the proof. 

\end{proof}

\smallskip

\begin{rmk}
In analogy with Remark~\ref{synt.rmkbig} the complex
$\frak{S}^I_{X_\bull}(r)_\etNis$ extends to an object in the global category
$D_\pro(\Sch_{W_\bull, \etNis})$.
The isomorphism in Proposition \ref{SI} extends to an isomorphism in $D_\pro(\Sm_{W_\bull,
  \etNis})$. Although the construction in the proof is valid only on the small site $X_{1,\etNis}$, the
isomorphism for different $X_\bull$ glue canonically.
So it follows that also the fundamental triangle in Theorem~\ref{syntF} extends to $D_\pro(\Sm_{W_\bull,
  \Nis})$.
\end{rmk}

\smallskip

\section{Connecting morphism in fundamental triangle}\label{conmorph}

\noindent
Let the notation be as in Section \ref{ftriangle}, in particular let
 $X_\bull$ be in $\Sm_{W_\bull}$. We assume $p> r$. 
The aim of this section is to show the following 
\begin{thm} \label{comp}
The connecting homomorphism 
\ga{}{\alpha: W_\bull \Omega^r_{X_1,
\log}[-r] \to  p(r)\Omega^{\le r-1}_{
X_\bull} \notag}
in the fundamental triangle (Theorem \ref{syntF}) is equal to
the composite morphism
\ga{}{\beta:  W_\bull \Omega^r_{X_1,
\log}[-r] \to W_\bull\Omega^{\ge r}_{X_1} \to
q(r)W_\bull\Omega^\bulle_{X_1}\xrightarrow{ {\rm Prop.~\ref{qis}}}
p(r)\Omega^\bulle_{X_\bull} \to   p(r)\Omega^{\le r-1}
_{X_\bull} \notag }
in $\D_\pro(X_1)$.
Here the non-labelled maps are the natural ones. 
\end{thm}
The theorem will imply the compatibility of $\alpha$ with the cycle class, see Section \ref{hodgeobstruction}.

\medskip

First of all we observe that it is enough to prove Theorem~\ref{comp} in \'etale topology,
i.e. that $\epsilon^*(\alpha) = \epsilon^*(\beta)$, because $\alpha= \tau_{\le r} ( \epsilon_*
\circ\epsilon^*( \alpha)) $ and $\beta = \tau_{\le r} ( \epsilon_* \circ\epsilon^*(
\beta)) $.

Definition~\ref{S} of $\frak{S}_{X_\bull}(r)_\et$ as a cone
gives a map $\frak{S}_{X_\bull}(r) \to J(r)\Omega^\bulle_{D_\bull}$ in
$\C_\pro(X_1)_\et$. Note that by Proposition~\ref{qis} there is a natural restriction quasi-isomorphism  
$ J(r)\Omega^\bulle_{D_\bull} \to \Omega^{\ge r}_{X_\bull} $.
We let $\kappa(r)$ be the composite map 
\ga{}{ \frak{S}_{X_\bull}(r)_\et \to  J(r)\Omega^\bulle_{D_\bull}
 \to \Omega^{\ge r}_{X_\bull} \quad \text{ in }  \C_\pro(X_1)_\et .\notag}
\begin{defn} \label{Sprime}
 We define $\frak{S}'_{X_\bull}(r)_\et=\cone (\frak{S}_{X_\bull}(r)_\et\xrightarrow{\kappa(r)}
\Omega^{\ge r}_{X_\bull})[-1]$ as an object in $\C_\pro(X_1)$.
\end{defn}

\medskip

The morphism $\Phi^J:  \frak{S}_{ X_\bull}(r) \to
W_\bull\Omega^r_{X_1, \log}[-r]$ in $\D_\pro(X_1)$ from Theorem~\ref{syntF} induces a morphism 
$\frak{S}'_{ X_\bull}(r) \to
W_\bull\Omega^r_{X_1, \log}[-r]$, still denoted by $\Phi^J$. 

We have a chain of isomorphisms in $\D_\pro(X)_\et$
\begin{equation}\label{con.chain}
\xymatrix{
\frak S'_{X_\bull}(r)_\et \ar[d]_-{(1)} \\
\cone (\frak S^I_{X_\bull}(r)_\et \to I(r) \Omega_{D_\bull}^\bulle  )[-1] \ar[d]_-{(2)} \\
\cone \big( \cone ( q(r)W_\bull\Omega^\bulle_{X_1} \xrightarrow{1-F_r}
W_\bull\Omega^\bulle_{X_1}   )[-1] \to q(r)W_\bull\Omega^\bulle_{X_1}  \big)[-1] \\
\frak E(r) :=  \cone ( W_\bull\Omega^\bulle_{X_1,\log}[-r] \to  q(r)W_\bull\Omega^\bulle_{X_1} )[-1]  \ar[u]^-{(3)} 
}
\end{equation}
where  $\rm (1)$ follows immediately from Definition \ref{Sprime}, $\rm (2)$ follows from
Proposition~\ref{qis} and $\rm (3)$ follows from Corollary \ref{syn.cor1}.

\begin{prop} \label{alpha}
 \begin{itemize}
  \item[(1)] In  $\D_\pro(X_1)_\et$, one has an exact triangle
\ga{}{    p(r)\Omega^\bulle_{
X_\bull}[-1] \to 
\frak{S}'_{ X_\bull}(r) \xrightarrow{\Phi^J} W_\bull\Omega^r_{X_1, \log}[-r]
\xrightarrow{[+1]}       \cdots                 \notag}
\item[(2)] In  $\D_\pro(X_1)_\et$, one has a commutative diagram of   
exact triangles
\[
\xymatrix{
q(r) W_\bull \Omega^\bulle_{X_1}[-1] \ar[r]  &  \frak E(r) \ar[r] &   W_\bull\Omega^r_{X_1,
  \log}[-r]  \ar[r]^-{[+1]} & \cdots \\
  p(r)\Omega^\bulle_{X_\bull}[-1] \ar[r] \ar[d] \ar[u]  & 
\frak{S}'_{ X_\bull}(r) \ar[r]^-{\Phi^J} \ar[d] \ar[u]_{(*)}  & W_\bull\Omega^r_{X_1, \log}[-r] \ar@{=}[u]
\ar[r]^-{[+1]} \ar@{=}[d] &       \cdots  \ar[d]  \ar[u]  \\ 
 p(r)\Omega^{< r}_{
X_\bull}[-1] \ar[r] & 
\frak{S}_{ X_\bull}(r) \ar[r]^-{\Phi^J} & W_\bull\Omega^r_{X_1, \log}[-r]
\ar[r]^-{[+1]}&       \cdots   
}
\]
where $\rm (*)$ is the composition of morphisms \eqref{con.chain}.
The upper triangle comes from the definition of $\frak E(r)$ as a cone and the
lower triangle is the fundamental triangle (Theorem \ref{syntF}).

\end{itemize}
\end{prop}
\begin{proof} 

For (1) we take the homotopy fibre of
the morphism of exact triangles
\[
\xymatrix{
p(r)\Omega^{< r}_{
X_\bull}[-1] \ar[r] \ar[d]_{d} & 
\frak{S}_{ X_\bull}(r) \ar[r]^-{\Phi^J} \ar[d]_{\kappa(r)} & W_\bull\Omega^r_{X_1, \log}[-r]
\ar[r]^-{[+1]} \ar[d] &       \cdots \ar[d]  \\
\Omega^{\ge r}  \ar@{=}[r] & \Omega^{\ge r}  \ar[r] & 0 \ar[r]^-{[+1]}&       \cdots 
}
\]
where the upper triangle is the fundamental triangle (Theorem~\ref{syntF}). 

We get an exact triangle in $\D_\pro(X_1)_\et$
\ga{}{   \cone  (p(r)\Omega^{<r}_{
X_\bull}[-1]\to \Omega^{\ge r}_{X_\bull})[-1] \to 
\frak{S}'_{ X_\bull}(r) \xrightarrow{\Phi^J}  W_\bull\Omega^r_{X_1, \log}[-r]
\xrightarrow{[+1]}       \cdots                 \notag}
and note that $\cone  (p(r)\Omega^{<r}_{
X_\bull}[-1]\to \Omega^{\ge r}_{X_\bull})$ is quasi-isomorphic to $p(r)\Omega_{X_\bull}^\bulle$.

Part (2) follows immediately via the isomorphisms \eqref{con.chain}.

\end{proof}

Theorem \ref{comp} follows now from Proposition \ref{alpha} together with \eqref{con.chain}.


\section{The motivic complex} \label{motiviccomplex}

\noindent

The aim of this
section is to define a motivic pro-complex of the $p$-adic scheme $X_\bull$ as an object
in $\D_\pro (X_1)_\Nis$. 
We shall show in Section~\ref{hodgeobstruction} that liftability of the cycle class to a
cohomology class of this complex precisely computes the obstruction
for the refined crystalline cycle class to be Hodge.

\medskip

We recall the definition of Suslin-Voevodsky's  cycle complex on the smooth scheme $X/k$
for an arbitrary field $k$,
following \cite[Defn.~3.1]{SV}. It is defined as an object $\Z(r)$ in the
abelian
category of complexes of abelian sheaves on the  big Nisnevich site $
\Sm/k$. Furthermore, it is a complex of sheaves with transfers. One has
\ga{SV}{\Z(r)=\sC^\bulle(\Z_{tr}(\G_m^{\wedge   r}))[-r].}
We explain what this means: We think of  $\G_m=\A^1\setminus \{0\}$ as a scheme. By $\Z_{tr}(X)$
we denote the presheaf with transfers defined by the formula $\Z_{tr}(X)(U)= {\rm Cor}(U,X)$,
for any $X\in \Sm/k$, where ${\rm Cor}(U,X)$ is the free abelian group generated  by
closed integral subschemes  $Z\subset U\times_k X$ which are finite and
surjective over a component of $U$ (\cite[Section~1]{SV}). Wedge product is defined as
$\Z_{tr}(\G_m^{\widehat \ r })=\Z_{tr}(\G_m^{\times r })/{\rm im}({\rm faces}) $, where the faces are defined by 
$(x_1,\ldots, x_{r-1})\mapsto  (x_1,\ldots,1,\ldots x_{r-1}) $. Finally, for
any presheaf of abelian groups $\sF$ on $\Sm/k$, one defines the simplicial
presheaf $\sC_\bulle(\sF)$ by $\sC_i(\sF)(U)=\sF(U\times \Delta^i)$. One sets
$\sC^i(\sF)=\sC_{-i}(\sF)$. So in sum, one has 
\ga{}{   \Z(r)^i(U)={\rm  Cor}(U\times_k \Delta^{r-i},   \G_m^{\widehat \  r}). \notag
 }

Clearly $\Z(r)$ is supported in degrees $\le r$. Its last
Nisnevich  cohomology sheaf is the Milnor $K$-sheaf 
\ga{Moritz}{ \sH^r(\Z(r))=\sK^M_r .} 
We refer to
\cite[Thm.~3.4]{SV} where it is computed for fields, and in general,  one needs
the Gersten resolution for Milnor $K$-theory on smooth varieties, established in  \cite{EVM},\cite{Ke} and unpublished work of Gabber. Note that in case the base field $k$ is
finite one has to use a refined version of the usual Milnor $K$-sheaves, defined in
\cite{Ke2}. See also Section \ref{localcomputations} for more details about the Milnor
$K$-sheaf. 
The essential property of this refined Milnor $K$-sheaf that we need, is that
it is locally generated by symbols $\{a_1 , \ldots , a_r \} $ with 
$a_i \in \mathcal O_{X}^\times $ ($1\le i\le r$).

\smallskip

For $X\in \Sm_k$ we denote by $\Z_X(r)$ the restriction of $\Z(r)$ to the small Nisnevich
site of $X$.
One has  from  \cite[Cor~19.2]{MVW} and \cite[Thm.~1.1]{Ke}
\ga{CH}{ H^{2r}(X, \Z_{X}(r))=H^r(X, \sK^M_{ X,r})=\CH^r(X).       }

\medskip

{From} now on the notation is as in Section \ref{conmorph}. In particular $X_\bull/ W_\bull$
is in $\Sm_{W_\bull}$ and $X_1= X \otimes_W k$. We assume $r<p$.

 We will consider
$\Z_{X_1}(r)$ as an object in $D(X_1)=D(X_1)_\Nis$ and
also as a constant pro-complex in $\D_\pro(X_1)=\D_\pro(X_1)_\Nis$.
So \eqref{Moritz} enables us to define the map 
\eq{7.4a}{\log: \Z_{X_1}(r)\to \sH^r(\Z_{X_1}(r))[-r]=\sK^M_{X_1,r}[-r]
\xrightarrow{d\log [ \ ]} W_\bull\Omega^r_{X_1,\log}[-r] }
in $\D_\pro(X_1)$, 
where $[ \ ]$ is the Teichm\"uller lift. 

\medskip

Recall that one has a map $\Phi^J: \frak{S}_{X_\bull}(r)\to
W_\bull\Omega^r_{X_1,\log}[-r]$ in $\D_\pro(X_1)=\D_\pro(X_1)_\Nis$ (Theorem~\ref{syntF})
with $\frak{S}_{X_\bull}(r)$ defined in
Definition~\ref{S}.

\begin{defn} \label{Z(r)}
We assume $p>r$.  We define the {\it motivic pro-complex}  $\Z_{X_\bull}(r)$ of $X_\bull$ as an
object in 
 $ \D_\pro(X_1)$ by 
\ga{}{ \Z_{X_\bull} (r) =\cone (   \frak{S}_{X_\bull}(r) \oplus
\Z_{X_1}(r)
\xrightarrow{\Phi^J\oplus -\log } W_\bull \Omega^r_{X_1, \log} [-r]   )[-1].
\notag}
\end{defn}

Note that by Lemma \ref{hom.unicity}, the cone is well defined up to unique isomorphism in the
triangulated category $\D_\pro(X_1)$. In fact the map 
\begin{equation}\label{mot.mildrw}
\mathcal H^r (\Z_{X_1}(r) ) =
\mathcal K^M_{X_1,r} \to W_\bull \Omega^r_{X_1, \log}
\end{equation}
is an epimorphism, since 
$W_\bull \Omega^r_{X_1, \log}$ is generated by symbols.

\smallskip

\begin{prop} \label{mot.basicprop} \mbox{}
\begin{itemize}
\item[(0)] One has $\Z_{X_\bull}(0)=\Z$, the constant sheaf $\Z$ in degree $0$.
\item[(1)]
One has $\Z_{X_\bull}(1)=\mathbb G_{m, X_\bull} [-1]$.
\item[(2)]
 The motivic complex  $\Z_{X_\bull}(r)$ has support in cohomological degrees $\le r$. For
$r\ge 1$, if the
Beilinson-Soul\'e conjecture is true, it has support in cohomological degrees $[1 , r]$.
\item[(3)] One has $\Z_{X_\bull}(r) \otimes^L_\Z  \Z/p^\bull = \mathfrak S_{X_\bull}(r)$
  in $\D_\pro(X_1)$. 
\item[(4)] One has $\sH^r(\Z_{X_\bull}(r))= \sK_{X_\bull,r}^M$ in $\Sh_\pro (X_1)$.
\item[(5)] There is a canonical product structure 
\[
\Z_{X_\bull}(r) \otimes^L_\Z  \Z_{X_\bull}(r')  \to \Z_{X_\bull}(r+r')
\]
compatible with the products on $\Z_{X_1}(r)$ and $\frak{S}_{X_\bull}(r)$.
\end{itemize} 
\end{prop}
\begin{proof} We show (0). One has $W_\bull\Omega^0_{X_1,\log} =\Z/p^\bull$,
$\Z_{X_1}(0)=\Z$ and for example by Theorem~\ref{syntF}, one has
$\frak{S}_{X_\bull}(0)=\Z/p^\bull$. So (0) is clear from Definition~\ref{Z(r)}. 

\medskip

We show (2).
For all $i\in \Z$, one has a long exact sequence
$$\ldots \to \sH^i(\Z_{X_\bull}(r)) \to \sH^i(\frak{S}_{X_\bull}(r))\oplus
\sH^i(\Z_{X_1}(r))\to \sH^i(W_\bull\Omega^r_{X_1,\log}[-r])\to \ldots$$ 
By Theorem \ref{syntF} the syntomic complex $\mathfrak S_{X_\bull}(r) $
has support in degrees $[1,r]$ for $r \ge 1$. The Beilinson-Soul\'e conjecture
predicts the same for the motivic complex $\Z_{X_1}(r)$. So (2) follows because
\eqref{mot.mildrw} is an epimorphism.

\medskip

We show (4). One has an exact sequence 
$$0 \to \sH^r(\Z_{X_\bull}(r)) \to \sH^r(\frak{S}_{X_\bull}(r))\oplus
\sH^r(\Z_{X_1}(r))\xrightarrow{\Phi^J\oplus -\log } W_\bull\Omega^r_{X_1,\log}\to
0$$ 
By Theorem~\ref{syntF}, one has an exact sequence 
 $$0\to p\Omega^{r-1}_{X_\bull}/p^2d\Omega^{r-2}_{X_\bull}\to
\sH^r(\frak{S}_{X_\bull}(r)) \xrightarrow{\Phi^J} W_\bull\Omega^r_{X_1,\log}\to 0  $$
which induces the upper row in the commutative diagram with exact rows  (the bottom row is
Theorem \ref{thm2.1})  
\[
\xymatrix{
0\ar[r]  &  p\Omega^{r-1}_{X_\bull}/p^2d\Omega^{r-2}_{X_\bull}\ar[r]  \ar@{=}[d]  &
 \sH^r( \Z_{X_\bull}(r) )  \ar[r] &     \sH^r( \Z_{X_1}(r) )  \ar[r] & 0 \\ 
0\ar[r] &  p\Omega^{r-1}_{X_\bull}/p^2d\Omega^{r-2}_{X_\bull}\ar[r]  &  \sK_{
X_\bull,r}^M   \ar[r]  \ar[u]^{(*)}  &  \sK_{X_1,r}^M\ar[r] \ar[u]_{\wr}   & 0
}
\]
Here the arrow $\rm (*)$ is induced by Kato's syntomic regulator map 
\cite[Sec.\ 3]{Ka2}.
By \eqref{Moritz}, the right vertical arrow is an isomorphism, so by the five-lemma $\rm
(*)$ is also an isomorphism.

\medskip

{From} (4) and (2) one deduces (1), since the Beilinson-Soul\'e vanishing is clear for $r=1$.

\medskip

We show (3). The sheaf $W_n\Omega^r_{X_1,\log}$ is a sheaf of flat $\Z/p^n$-modules, so
$$W_\bull\Omega^r_{X_1,\log} \otimes_\Z^L \Z/p^\bull = W_\bull\Omega^r_{X_1,\log}  \quad
\text{ in } \quad \D_\pro(X_1).$$
By Theorem \ref{syntF} this also implies that $\mathfrak S_{X_\bull}(r) \otimes_\Z^L
\Z/p^\bull = \mathfrak S_{X_\bull}(r)  $. Geisser-Levine show that $\Z_{X_1}(r)
\otimes^L_\Z \Z/p^n = W_n\Omega^r_{X_1,\log}[-r] $, see \cite{GL}. So from the definition
of $\Z_{X_\bull}(r)$ we conclude that  $\Z_{X_\bull}(r) \otimes^L_\Z  \Z/p^\bull =
\mathfrak S_{X_\bull}(r)$.

We show (5). By a simple argument analogous to the proof of Lemma~\ref{hom.unicity} having
a product morphism as in (5) is equivalent to having two morphisms
\begin{align*}
\Z_{X_\bull}(r) \otimes^L_\Z  \Z_{X_\bull}(r') &\to \Z_{X_1}(r+r')\\
\Z_{X_\bull}(r) \otimes^L_\Z  \Z_{X_\bull}(r') &\to \frak{S}_{X_\bull}(r+r')
\end{align*}
in $D_\pro(X_1)$, which become equal when composing with the maps to $W_\bull
\Omega^{r+r'}_{X_1,\log}[-r]$. We let the two morphisms be induced by the usual product of the
Suslin-Voevodsky motivic complex and the product on the syntomic complex.
\end{proof}

\begin{prop}[Motivic fundamental triangle] \label{diag}
 One has a unique commutative diagram of exact triangles in $\D_\pro(X_1)$
\ga{}{\xymatrix{
\ar@{=}[d] p(r)\Omega^{<r}_{X_\bull}[-1] \ar[r]& \ar[d]  \Z_{X_\bull}(r)
\ar[r] &
\ar[d] \Z_{X_1}(r) \ar[d]^{d\log} \ar[r] & \cdots \ar[d]\\
p(r)\Omega^{<r}_{X_\bull}[-1] \ar[r] & \frak{S}_{X_\bull}(r) \ar[r]&
W_\bull\Omega^r_{X_1,\log}[-r]\ar[r] & \cdots
}
\notag} 
where the bottom exact triangle comes from Theorem~\ref{syntF} and the maps in the right square are the
canonical maps.
\end{prop}

\begin{proof}
The square
\[
\xymatrix{
\ar[d]  \Z_{X_\bull}(r)
\ar[r] &
 \Z_{X_1}(r) \ar[d]^{d\log}\\
\frak{S}_{X_\bull}(r) \ar[r]&
W_\bull\Omega^r_{X_1,\log}[-r]
}
\]
is homotopy cartesian by definition. So the existence of the commutative diagram in the
proposition follows from \cite[Lemma 1.4.4]{Ne}.

For uniqueness one has to show that the morphism
\[
 p(r)\Omega^{\le r-1}_{X_\bull}[-1] \to  \Z_{X_\bull}(r)
\]
is uniquely defined by the requirements of the proposition. This can be shown analogously
with Lemma~\ref{hom.unicity}.

\end{proof}

\begin{cor}\label{projbuniso}
 For $Y_\bull = X_\bull \times \P^m$
one has a projective bundle isomorphism
\[
\bigoplus_{s=0}^m H^{r'-2s}_\cont(X_1, \Z_{X_\bull} (r-s) ) \xrightarrow{\oplus_s {\rm
    c}_1( \mathcal O(1))^s } H^{r'}_\cont(Y_1 , \Z_{Y_\bull}(r) )
\]
\end{cor}

\begin{proof}
By Proposition~\ref{diag} one has to show that the analogous maps for Suslin-Voevodsky
motivic cohomology of $X_1$ and for
Hodge cohomology are isomorphisms. This holds by \cite[Cor.~15.5]{MVW} and \cite[Exp.~XI, Thm.~1.1]{SGA7}.
\end{proof}

\section{Crystalline  Hodge obstruction and  motivic
complex}  \label{hodgeobstruction}

\noindent
Let the notation be as in Section \ref{conmorph}.
We additionally assume in this section that $X_1/k$ is proper.

Our goal in this section is to study a cohomological deformation condition for a rational equivalence
class  $\xi_1 \in  \CH^r(X_1)=H^{2r}(X_1,\Z_{X_1}(r))$ to lift to a cohomology class
$\xi\in H^{2r}_\cont(X_1,\Z_{X_\bull}(r)) $, where $\Z_{X_\bull}(r)$ is the motivic
complex defined in Section~\ref{motiviccomplex}. In fact we suggest to interpret the latter
group as the codimension $r$ cohomological Chow group of the formal scheme $X_\bull$.
\begin{defn} \label{contChowgr}
We define the {\it continuous Chow group} of $X_\bull$ to be  
$$\CH^r_\cont(X_\bull) = H^{2r}_\cont(X_1,\Z_{X_\bull}(r)). $$ 
\end{defn}
For the definition of continuous cohomology see Definition~\ref{hoto.contcoho}.
The deformation problem can be understood by means of the fundamental exact triangle in
Proposition~\ref{diag}, which gives rise to the exact obstruction sequence
\begin{equation}\label{obs.seq}
\CH^r_\cont(X_\bull) \to \CH^r(X_1)  \xrightarrow{\rm Ob} H^{2r}_\cont(X_1,p(r)\Omega^{<r}_{X_\bull}).
\end{equation}
We will compare the obstruction ${\rm Ob}(\xi_1)$ to the cycle class of $\xi_1$ in
crystalline and de Rham cohomology.

\medskip

Note that by general homological algebra (formula \eqref{hom.lim1}) we have an exact sequence
\[
0 \to \varprojlim_n{\!}^1 H^{2r-1}(X_1, \Z_{X_n}(r) )\to \CH^r_\cont (X_\bull ) \to
 \varprojlim_n H^{2r}(X_1, \Z_{X_n}(r) ) \to 0.
\] 
In particular by Proposition~\ref{mot.basicprop}(1) and the vanishing of
$\varprojlim_n{\!\!\!\!\!\!}^1 H^0(X_1,\G_{m,X_n})$  we get an isomorphism
\begin{equation}\label{chowpic}
\CH^1_\cont (X_\bull) \xrightarrow{\sim} \varprojlim_n \Pic(X_n).
\end{equation}
Note that if $X_\bull$ is the $p$-adic formal scheme associated to the smooth projective
scheme $X/W$ there is an algebraization isomorphism \cite[Thm.~5.1.4]{EGA3}
\begin{equation} \label{picalg}
\Pic(X) \xrightarrow{\sim} \varprojlim_n \Pic(X_n).
\end{equation}
The relation of $\CH^r_\cont(X_\bull)$ to formal systems of vector bundles is explained in
Section~\ref{cherniso}. Unfortunately, an analog of the algebraization
isomorphism~\eqref{picalg} is unknown.

\medskip

We first recall the construction of the crystalline cycle class, as given by
Gros \cite[II.4]{G} and Milne \cite[Section~2]{Mi}, using the Gersten
resolution for $W_\bull\Omega^r_{X_1,\log}$  \cite[(0.1)]{GS} and the Gersten
resolution for the Milnor $K$-sheaf  $\sK_r^M$ \cite[Thm.~1.1]{Ke}. The morphism
$d\log \circ [ \ ]: \sK_{X_1,r}^M\to W_\bull\Omega^r_{X_1,\log} $  maps the Gersten
resolution for $\sK_{X_1,r}^M$ to the one for $W_\bull\Omega^r_{X_1,\log} $, where
$[ -
]$ is the Teichm\"uller lift. Thus, for any integral codimension $r$ subscheme
$Z\subset X_1$, one obtains as a consequence of purity 
\ga{}{\Z\cdot [Z]=H^r_Z(X_1, \sK^M_r)\xrightarrow{d \log   } 
\Z/p^\bull\cdot [Z]=H^r_Z(X_1, W_\bull\Omega^r_{X_1,\log}),\notag}
where the map $\Z\to \Z/p^n$ is just the projection.  
The image of $$1 \cdot [Z]\;\;\; \text{ in } \;\;\; H^r_\cont(X_1, W_\bull\Omega^r_{X_1,\log}),$$ after forgetting
supports, is the cycle class of $Z$.
By $\Z$-linear extension, Gros and Milne define the cycle class map
$$\CH^r(X_1)  \to H^r_\cont(X_1,W_\bull\Omega^r_{X_1,\log} ). $$
Also we observe that the cycle class map is induced, via the Bloch formula \cite{Ke}
\[
\CH^r(X_1) =H^r(X_1, \sK^M_r ),
\] 
 by the morphism of pro-sheaves $ \sK^M_{X,r} \to W_\bull\Omega^r_{X_1,\log} $.

On the other hand, one has a natural map of complexes 
\begin{equation} \label{crho.map1}
 W_\bull\Omega^r_{X_1,\log} [-r]\to W_\bull\Omega^{\ge r}_{X_1} \to q(r)
W_\bull\Omega^\bulle_{X_1}          
\end{equation}
 in $\C_\pro(X_1)$.

\begin{defn} \label{crystcycleclass}
 For $\xi\in \CH^r(X_1)$, its
{\it refined crystalline cycle class} is the class $$c(\xi) \in H^{2r}_\cont(X_1,
q(r)W_\bull\Omega^r_{X_1})$$
induced by \eqref{crho.map1}.

The {\it crystalline cycle class} of $\xi$ is the image $c_{{\rm
cris}}(\xi)$ of $c(\xi)$ in 
$H^{2r}_\cont(X_1, W_\bull\Omega^\bulle_{X_1})$.

\end{defn}

By abuse of notation we make the identifications
\begin{align*}
H^i_\cont(X_1,q(r)W_\bull\Omega^\bulle_{X_1} ) &=
H^i_\cont(X_1,p(r)\Omega^\bulle_{X_\bull} )\\
 H^i_\cont(X_1,W_\bull\Omega^\bulle_{X_1} ) &=
H^i_\cont(X_1,\Omega^\bulle_{X_\bull} )
\end{align*}
using the comparison isomorphism from \eqref{cris.diaco} and Proposition~\ref{qis}.

\begin{defns} \label{Hodge}\mbox{}
\begin{itemize}
\item[(1)] One says that the crystalline (resp. refined crystalline) cycle class
of $\xi$ is {\it
Hodge} if and only if $c_{{\cris}}(\xi)$ (resp. $c(\xi)$)
lies in the image of  $
H^{2r}_\cont(X_1, \Omega^{\ge r}_{X_\bull})$ in $H^{2r}_\cont(X_1,
\Omega^\bulle_{X_\bull})$ (resp. in   $H^{2r}_\cont(X_1,
p(r)\Omega^\bulle_{X_\bull})).$
\item[(2)] One says that  $c_{{\cris}}(\xi)$ is {\it
Hodge modulo torsion} if and only if $c_{{\cris}}(\xi)\otimes \Q$ lies
in the image of  $
H^{2r}_\cont(X_1, \Omega^{\ge r}_{X_\bull}) \otimes \Q$ in $H^{2r}_\cont(X_1,
\Omega^\bulle_{X_\bull})\otimes \Q $.
\end{itemize}
\end{defns}

\begin{rmks}  \label{littlermks}  \mbox{}
\begin{itemize}
\item[(1)] By the degeneration of the Hodge-de Rham spectral sequence modulo torsion, the map
$H^{2r}_\cont(X_1, \Omega^{\ge r}_{X_\bull}) \otimes \Q\to H^{2r}_\cont(X_1,
\Omega^\bulle_{X_\bull})\otimes \Q $ is injective.

\item[(2)] If $H^b_\cont(X_1, \Omega^a_{X_\bull})$ is a torsion-free $W(k)$-module for all
$a,b
\in \N$, then the composite map 
$$H^{2r}_\cont(X_1, \Omega^{\ge r}_{X_\bull})\to H^{2r}_\cont(X_1, p(r)\Omega^\bulle_{X_\bull})\to
H^{2r}_\cont(X_1, \Omega^\bulle_{X_\bull})$$ is injective, and thus the left map as well.
\item[(3)]
 The map $H^{2r}_\cont(X_1, p(r)\Omega^{\ge r}_{X_\bull}) \otimes \Q\to H^{2r}_\cont(X_1,
\Omega^\bulle_{X_\bull})\otimes \Q $ is an isomorphism.

\end{itemize}
\end{rmks}

\medskip

Now we formulate one of our main theorems:  
\begin{thm}\label{def.thm}
 Let $X_\bull/W_\bull$ be a smooth projective $p$-adic formal scheme. Let $\xi_1 \in
 \CH^r(X_1)$ be an algebraic cycle class.  Then
\begin{itemize}
\item[(1)] its refined crystalline class
$c(\xi_1)\in H^{2r}_\cont(X_1, q(r) W_\bull\Omega^\bulle_{X_1})$ is Hodge 
if and only if $\xi_1$ lies in the image of the restriction map
$
\CH_\cont^r(X_\bull) \to \CH^r(X_1),
$ 
\item[(2)]  its crystalline class $c_{{\rm cris}}(\xi_1)\in H^{2r}_\cont(X_1,
W_\bull\Omega^\bulle_{X_1})$ is Hodge modulo torsion if and only if $\xi_1\otimes
\Q$ lies in the image of the restriction map
$
\CH_\cont^r(X_\bull) \otimes \Q \to \CH^r(X_1)\otimes \Q.
$ 
\end{itemize} 
\end{thm}
\begin{proof}
 The second part follows from the first one and Remark~\ref{littlermks}(3). 
For (1) we observe that we have a commutative diagram with exact rows, extending \eqref{obs.seq},
\[
\xymatrix{
\CH_\cont^r(X_\bull) \ar[r] \ar[d]_c & \CH^r(X_1)  \ar[r]^-{{\rm Ob}}  \ar[d]_c & H^{2r}_\cont
(X_1, p(r)\Omega_{X_\bull}^{<r} )  \ar@{=}[d]\\
 H^{2r}_\cont (p(r) \Omega^{\ge r}_{X_\bull} ) \ar[r]  & H^{2r}_\cont (p(r) \Omega^{\bulle}_{X_\bull} ) \ar[r] & H^{2r}_\cont (p(r) \Omega^{<r}_{X_\bull} ) 
}
\]
Indeed, the right square commutes by Theorem~\ref{comp}.
The theorem follows by a simple diagram chase.
\end{proof}

\begin{rmk}
For $r=1$ Theorem~\ref{def.thm} is due to Berthelot-Ogus \cite{BO2}, relying on a
construction of a complex similar to our $\frak S'_{X_\bull}(1)$ which was first studied in \cite[p.~124]{DelK3}. Note the
identification \eqref{chowpic} of $\CH^1_\cont(X_\bull)$ with the Picard group.
\end{rmk}

\section{Continuous $K$-theory and Chern classes} \label{chern}

\noindent 
The aim of this section is firstly to describe Quillen's $+$-construction and $\rm Q$-construction for $K$-theory of
the $p$-adic formal scheme $X_\bull$ in $\Sch_{W_\bull}$.
Secondly, we show 
$$\bigoplus_{r<p} H^{2r}_\cont(\BGL_{W_1} , \Z_{\BGL_{W_\bull}}(r)  )  = \Z[c_1,c_2,\ldots
]_{<2p} $$
where the right side is the degree smaller than $2p$ part of  polynomial ring in the
univeral Chern classes $c_r$. The latter have (cohomological) degree $2r$.
By pullback we get Chern classes in motivic cohomology for continuous
higher $K$-theory for smooth  $X_\bull$.

\smallskip

Let now $X_\bull $ be in $\Sch_{W_\bull}$. 
\begin{defn}\label{kth.def}
By $K_{X_\bull}\in \S_\pro(X_1)$ we denote the pro-system of simplicial presheaves given by
Quillen's $\rm Q$-construction. Explicitly, for $U_\bull \to X_\bull $ \'etale $K_{X_\bull}(U_1)$ is given by
\[
n \mapsto \Omega\, {\rm B}\, {\rm Q\, Vec} (U_n) \quad\quad (n\ge 1),
\]
where ${\rm Vec} (U_n)$ is the exact category of vector bundles on $U_n$, $\rm Q$ is
Quillen's $\rm Q$-construction functor and $\rm B$ is the classifying space functor, see
\cite[Sec.\ 5]{Srin}. 
\end{defn}

\begin{defn}
Continuous $K$-theory of $X_\bull$ in $\Sch_{W_\bull}$ is defined by
\[
K^\cont_i(X_\bull) = [S^i_{X_1} , K_{X_\bull}],
\]
where $S^i_{X_1}$ is the constant presheaf pro-system of the simplicial $i$-sphere in $\S_\pro(X_1)$.
\end{defn}

By \cite[Sec.~IX.3]{BoK} (see Proposition~\ref{hoto.expfib}) there is a short exact sequence
\[
0\to \varprojlim_n{\!}^1K_{i+1}(X_n) \to K^\cont_i(X_\bull) \to \varprojlim_n K_i(X_n) \to 0.
\]

Thomason-Throbaugh \cite[Sec.~10]{TT} show that $K_{X_\bull}$ satisfies Nisnevich descent.

\begin{prop}\label{kth.descent}
The $K$-theory presheaf of Definition~\ref{kth.def} satisfies Nisnevich descent in the
sense of Definition~\ref{hoto.descent}.
\end{prop}

In particular from Lemma~\ref{hoto.descentBK} we get a Bousfield-Kan descent spectral sequence
\begin{equation}\label{kth.desceq}
E^{s,t}_2 = H^{s}_\cont (X_1, \sK_{X_\bull,t} ) \Longrightarrow   K_{t-s}^\cont (X_\bull)
\quad \quad t \ge s .
\end{equation}
where $\sK_{X_\bull,t}$ is the pro-system of Nisnevich sheaves of homotopy groups of $K_{X_\bull}$.

Our aim in the rest of this section is to construct a Chern character from continuous
$K$-theory to continuous motivic cohomology.
\begin{defn}
By $\BGL_{m,R}$  ($m\ge 1$)  we denote the simplicial classifying scheme
\[
\xymatrix{
 \cdots & GL_{m,R} \times GR_{m,R}  \ar@<10pt>[r] \ar@<6pt>[r] \ar@<2pt>[r]  &  GL_{m,R} \ar@<6pt>[r]
 \ar@<2pt>[r] \ar@<2pt>[l] \ar@<6pt>[l] &  \{ * \} \ar@<2pt>[l]
}
\]
of the general linear group
over the base ring $R$. By $\BGL_R$ we denote the ind-simplicial scheme 
\[
\cdots  \to\BGL_{m,R} \to \BGL_{m+1,R}  \to \BGL_{m+2,R} \to \cdots 
\] 
\end{defn}

In the usual way one can associate to $\BGL_R$ its small \'etale and Nisnevich sites,
denoted by $\BGL_{R,\et}$ and $\BGL_R=\BGL_{R,\Nis}$. 

The following facts are well known to the experts:
\begin{itemize}
\item[(a)]
There is a canonical isomorphism
\begin{equation}
\bigoplus_r H^{2r}( \BGL_k , \Z_{\BGL_k}(r)  ) = \Z [ c_1,c_2 ,\ldots ],
\end{equation}
where the $c_i$ are Chern classes of the universal bundle on $\BGL_{n,k}$  of cohomoloical
degree $2i$, see \cite[Lem.~7]{Push}.
\item[(b)]
There is a canonical isomorphism
\begin{equation}
\bigoplus_r H^{r}_\cont( \BGL_k ,\oplus_t \Omega_{\BGL_{W_\bull}}^t[-t]  ) = W [ c_1,c_2 ,\ldots ],
\end{equation}
where the $c_i$ are Chern classes of the universal bundle on $\BGL_{n,k}$  of cohomoloical
bi-degree $(r,t)=(2i,i)$, see Thm.~1.4 and Rmk.~3.6 of \cite{G}.
\end{itemize}

{From} the Hodge-de Rham spectral sequence and (b) we deduce that 
\begin{align*}
H^{2r-1}_\cont (\BGL_k ,p(r) \Omega^{<r}_{BGL_{W_\bull}}) &= 0,\\
H^{2r}_\cont (\BGL_k , p(r)\Omega^{<r}_{BGL_{W_\bull}}) &= 0.
\end{align*}
By the fundamental triangle in Proposition~\ref{diag} this implies that
\[
\bigoplus_{r<p} H^{2r}_\cont( \BGL_k , \Z_{\BGL_{W_\bull}}(r)  ) \xrightarrow{\sim}  \bigoplus_{r<p} H^{2r}( \BGL_k , \Z_{\BGL_k}(r)  )
\]
is an isomorphism. We conclude:
\begin{prop}\label{kth.bgl}
There is a canonical isomorphism of graded groups
\[
\bigoplus_{r<p}  H^{2r}_\cont(\BGL_{W_1} , \Z_{\BGL_{W_\bull}}(r)  )  = \Z[c_1,c_2,\ldots ]_{<2p},
\]
where the universal Chern classes $c_i$ live in cohomological degree $2i$. The index $2p$
on the right side means that we take only sums of monomials of degree less than $2p$. 
\end{prop}

\medskip

By the construction of Gillet \cite{Gil} the universal Chern class $c_r$ of
Proposition~\ref{kth.bgl} leads to a morphism $${\rm c}_r\in [\BGL_{X_\bull} , \rK
\,\Z_{X_\bull}(r)[2r] ]$$
in the homotopy category ${\rm hS}_\pro(X_1)$, see Notation~\ref{hoto.homcat}. 
Here $\rK$ stands for the Eilenberg-MacLane functor of Proposition~\ref{hoto.adj} and
$\BGL_{X_\bull}$ is the natural pro-system of presheaves of simplicial sets on
$X_{1,\Nis}$ given on $U_n\to X_n$ \'etale by $\varinjlim_m\BGL_{W_n,m}(U_n) $.    
By Proposition~\ref{kth.descent} and a functorial version of Quillen's $+=\rm Q$ theorem (see the
proof of Prop.~2.15 of
\cite{Gil}) there is a canonical isomorphism $$K_{X_\bull} \cong\Z\times \Z_\infty
\BGL_{X_\bull} $$
in ${\rm hS}_\pro(X_1)$, where
$\Z_\infty$ is the Bousfield-Kan $\Z$-completion functor \cite{BoK}. Completion therefore induces a map
\[
[\BGL_{X_\bull} , \rK \,\Z_{X_\bull}(r)[2r] ] \to [K_{X_\bull} , \rK \,\Z_{X_\bull}(r)[2r] ] 
\]
and for $r<p$ we get continuous Chern class maps
\begin{equation}
{\rm c}_r : K_i^\cont(X_\bull) \to H^{2r-i}_\cont (X_1, \Z_{X_\bull}(r) ),
\end{equation}
which are group homomorphisms for $i>0$ and satisfy the Whitney formula for $i=0$. 

The degree $r$ part of the universal Chern character is a universal polynomial $\ch_r\in
\Z[1/r!] [c_1 , \dots ]$.
As above by pullback we get Chern characters
\begin{equation}
{\rm ch}_r : K^\cont_i (X_\bull) \to H^{2r-i}_\cont (X_1 , \Z_{X_\bull}(r))_{
  \Z[\frac{1}{r!}]}, 
\end{equation}
which are additive and compatible with product. The lower index $\Z[\frac{1}{r!}]$ stands for
$ - \otimes_\Z \Z[\frac{1}{r!}]$.
 Note that the canonical morphism
\[
H^{2r-i}_\cont (X_1, \Z_{X_\bull}(r) )_{\Z[ \frac{1}{r!}] } \xrightarrow{\sim}   H^{2r-i}_\cont (X_1 , \Z[\frac{1}{r!}]_{X_\bull}(r))
\]
is an isomorphisms for $r<p$, as follows from Proposition~\ref{diag}.

\section{Results from topological cyclic homology} \label{topcyc}

\noindent
We summarize some deep results about $K$-theory which are proved using the theory of topological
cyclic homology, due to
McCarthy, Madsen, Hesselholt, Geisser and others. Note that we state results not in their general
form, but in a form sufficient for our application.

In this section we work in \'etale topology only, i.e.\ all sheaves and cohomology groups are
in \'etale topology. The prime $p$ is always assumed to be odd.

Let $R$ be a discrete valuation ring, finite flat over $W$ and write $R_n = R/(p^n)$.
Let $X$ be in $\Sm_R$ and $X_\bull$ be the associated $p$-adic
formal scheme in $\Sm_{R_\bull}$, i.e.\ $X_n=X\otimes_R R_n$. Denote by $i:X_\red \hookrightarrow X$
the immersion of the reduced closed
fibre and by $j:X_K \to X$ the immersion of the general fibre, $K={\rm frac}( R)$.
Using the arithmetic square \cite[Sec.~VI.8]{BoK} and the theorems of McCarthy \cite{Mc} and Goodwillie \cite{Good}, Geisser-Hesselholt \cite[Thm.~A]{GH1} deduce results about integral $K$-theory in
the relative affine situation $X_\red \hookrightarrow X_n$. 
Combining their result with Thomason's Zariski descent for $K$-theory,
Proposition \ref{kth.descent}, in order to reduce to affine $X_n$ and \'etale decent
for topological cyclic homology \cite[Cor.~3.3.3]{GH2} we get: 

\begin{prop}\label{topcyc.desc} \mbox{}
\begin{itemize}
\item[(a)] The relative $K$-groups $K_s(X_n,X_{\red})$ are $p$-primary torsion of finite
  exponent for any $n\ge 1$, $s\ge 0$.
\item[(b)] The presheaf of simplicial sets $K_{X_n,X_\red}$ on the small \'etale site of $X_\red$
  satisfies \'etale descent, see Definition~\ref{hoto.descent}.
\end{itemize}
\end{prop}

Generalizing the work of Suslin and Panin, Geisser-Hesselholt \cite{GH3} obtain the
following continuity result for
$K$-theory with $\Z/p$-coefficients.
Let $(\sK/p)_{X,s}$ be the sheafification in the \'etale topology of $X$ of $K$-groups with $\Z/p$-coefficients and
let similarly $(\sK/p)_{X_\bull,s}$ be the pro-system of $K$-sheaves of the schemes $X_n$ on the \'etale site of $X_\red$.

\begin{prop}\label{topcyc.cont}
The restriction map induces an isomorphism of pro-systems of \'etale sheaves on $X_\red$
\[
i^* (\sK/p)_{X,s}  \xrightarrow{\sim} (\sK /p)_{X_\bull,s}.
\]
\end{prop}

Note that one also has a continuity isomorphism 
\begin{equation}
i^* \G_{m,X} \otimes^L_\Z \Z/p  \xrightarrow{\sim} \G_{m,X_\bull}  \otimes^L_\Z \Z/p
\end{equation}
in $D_\pro(X_\red)_\et$.

\medskip

In the rest of this section we study the relation of $K$-theory to a form of $p$-adic
vanishing cycles. 

\begin{defn}
We define 
\[
\frak{V}_X(r) =\cone ( \tau_{\le r}\,  R\, j_* \Z/p (r) \xrightarrow{{\rm res}} i_*\Omega^{r-1}_{X_\red,\log}[-r] )[-1],
\]
where $\rm res$ is the residue map of Bloch-Kato \cite[Thm.~1.4]{BK}.
\end{defn}

Note that the cone in the definition is unique up to unique isomorphism by Lemma~\ref{hom.unicity}.
There is a canonical product structure
\begin{equation}\label{sec.topcyc.prod}
\frak{V}_X(r_1) \otimes^L_{\Z/p}  \frak{V}_X(r_2)  \to \frak{V}_X(r_1+r_2).
\end{equation}

\smallskip

\begin{lem}\label{topcyc.gmiso}
The symbol map induces an isomorphism
\[
\G_{m,X} \otimes^L_\Z \Z/p [-1] \xrightarrow{\sim} \frak V_X(1)
\]
in $\D(X)_\et$.
\end{lem}

\begin{proof}
We have a short exact sequence of \'etale sheaves 
\[
0\to  \G_{m,X} \to j_* \G_{m, X_K} \to i_* \Z  \to 0.
\]
Forming the derived tensor product of the associated exact triangle in $\D(X)_\et$ with
$\Z/p$ and using the isomorphism
\[
j_* \G_{m, X_K} \otimes^L_\Z \Z/p  \xrightarrow{\sim}  \tau_{\le 1} R\, j_* \Z/p (1),
\]
we finish the proof of the lemma.
\end{proof}

Assume that $R$ contains a primitive $p$-th root of unity.
We have the following chain of isomorphisms of pro-systems of \'etale sheaves on $X_\red$:
\begin{equation}\label{topcyc.compiso}
i^*(\sK /p )_{X,s} \xrightarrow{\rm tr} i^*(\mathcal{TC}^\bull /p )_{X,s}
\xrightarrow{(*)} \bigoplus_{r\le s } i^* \sH^{2r-s}(\frak V_X(r)).
\end{equation}
Here $\rm tr$ is the B\"okstedt-Hsiang-Madsen trace \cite{BHM} from the \'etale $K$-sheaf
to the \'etale pro-sheaf of topological cyclic homology.
The map $\rm tr$ is an isomorphism by \cite[Thm.~B]{GH3}.
The isomorphism $(*)$ is the composite of isomorphisms induced by
\cite[Thm.~E]{HM} and \cite[Thm.~A]{GH4}.

Fix a primitive $p$-th root of unity $\zeta$ in $R$.
Recall that the Bott element $$\beta \in K_2(\Z_p [\zeta];\Z/p)$$ is the unique element which maps to
$\{\zeta\} \in K_1(\Z_p [\zeta];\Z/p)$ under the Bockstein. Uniqueness of this
Bott element follows from Moore's theorem \cite[App.]{Mil}, which says that
\[
K_2(\Z_p [\zeta ] )  = \Z/p \oplus (\text{divisible}) .
\]

 The Bott element  
\begin{equation}\label{topcyc.h0}
\beta \in H^0( \Spec
W[\zeta] ,\frak V(1)) \xleftarrow{\sim}  \ker (\G_m( W[\zeta]) \xrightarrow{p} \G_m( W[\zeta])  ) =  \zeta^\Z
\end{equation}
is by definition the element induced by $\zeta$, where the first isomorphism in  \eqref{topcyc.h0} is coming from Lemma~\ref{topcyc.gmiso}.

The composite isomorphism \eqref{topcyc.compiso} can be uniquely characterized as follows:
\begin{prop}\label{topcyc.isoprop}
If $R$ contains the $p$-th roots of unity there is a unique morphism
\[
i^*(\sK /p )_{X,s}  \xrightarrow{\sim}  \bigoplus_{r\le s } i^* \sH^{2r-s}(\frak V_X(r))
\]
of \'etale sheaves
mapping the local section $\beta^t \{ a_1 ,\ldots ,a_{s-2t} \} $ in $K$-theory  to the
 local section  $\beta^t \{ a_1 ,\ldots ,a_{s-2t} \} $ in $p$-adic vanishing cycles. Here
 $a_u$ ($1\le u\le s-2t$) are local sections of $i^* j_* \sO_{X_K}^\times$
for $t>0$ and  local sections of  $i^*  \sO_{X}^\times$ for $t=0$. This morphism is an isomorphism.
\end{prop}

\begin{proof}
The local sections  $\beta^t \{ a_1 ,\ldots ,a_{s-2t} \} $ on both sides are well-defined
by means of the product structure on $K$-theory  and the
product structure \eqref{sec.topcyc.prod} on $p$-adic vanishing cycles. In order to deduce
the proposition one has to note that the isomorphism constructed above is compatible with products
and that the target ring of the isomorphism is generated by the above Bott-type symbols \cite[Thm.~1.4]{BK}.
In fact the B\"okstedt-Hsiang-Madsen trace is compatible with products.
This is shown in \cite[Sec.~6]{GH2}.
\end{proof}

\section{Chern character isomorphism}\label{cherniso}

\noindent
In this section we show that under suitable hypotheses our Chern character from continuous
$K$-theory to continuous motivic cohomology of a smooth $p$-adic formal scheme is an
isomorphism. Using descent we firstly reduce it to an \'etale local problem with
$\Z/p$-coefficients. Secondly, we use the fact, Proposition~\ref{topcyc.isoprop}, that
there is some \'etale local isomorphism, which we show is the same as our Chern character.

\medskip

Consider a smooth $p$-adic formal scheme $X_\bull\in \Sm_{W_\bull}$ and let $d=\dim(X_1)$. The continuous $K$-group
$K_0^\cont (X_\bull)$ was defined in Section~\ref{chern}, as well as the Chern character
map to continuous motivic cohomology. 

\begin{thm} \label{ciso.isothm}
For $p>d+6$ the Chern character
\[
\ch : K^\cont_0(X_\bull)_\Q  \to \bigoplus_{r\le d} \CH_\cont^r(X_\bull)_\Q
\]
is an isomorphism.
\end{thm} 

Note that we have $\CH_\cont^r(X_\bull)=0$ for $r> d$ by Proposition~\ref{diag} and the
fact that
\[
\varprojlim_n\!\!^1 H^*(X_1, p(r)\Omega^{<r}_{X_n}) =0,
\]
because it is a pro-system of $W_n$-modules of finite length and therefore a
Mittag-Leffler pro-system.

 there is no
$\lim\!\!^1$-contribution to continuous Hodge cohomology. Indeed, the Hodge cohomology
group $H^*(X_n,\Omega^*_{X_n})$ is a $W_n$-module of finite length and so the pro-system
is Mittag-Leffler.

\begin{proof}
For $r +1< p$ we have a commutative diagram
\[
\xymatrix{
K_{1}^\cont( Y.)_\Q  \ar[r]^-{\ch_r} \ar@<1ex>[d]^{\partial}  &   H_{ \cont }^{2r+1} (Y_1 , \Z_{Y.}(r+1) )_\Q  \ar@<1ex>[d]^{\partial} \\
K_{0}^\cont( X.)_\Q  \ar[r]_-{\ch_r}  \ar[u]^{\{ T\}} &   H_{ \cont}^{2r}(X_1 , \Z_{X.}(r) )_\Q \ar[u]^{\{ T\}}
}
\]
where $Y_\bull=X_\bull \times \G_{m}$ and $T$ is a torus parameter. The maps $\partial$ in
the diagram are constructed in the standard way by the projective bundle formula for
$X_\bull \times \P^1$ and the
Mayer-Vietoris exact sequence, see Corollary~\ref{projbuniso} and \cite[Sec.~6]{TT}.
With the appropriate sign convention we get $\partial \circ \{ T\}  = \rm id$.

By the diagram it suffices to show that
\[
\ch:K_{1 }^\cont(Y.) _ \Q  \to \bigoplus_{r\le d+2} H_{\cont}^{2r-1}(Y_1 , \Z_{Y.}(r) )_\Q .
\]
is an isomorphism.

The Chern character induces a morphism on relative theories and so we obtain   a
commutative diagram with exact sequences 
\begin{align}\label{ci.relseq} &
\xymatrix{ 
K_2(Y_1)_\Q \ar[r] \ar[d]_{\ch}^{(1)}  &  K_1(Y_\bull, Y_1)_\Q   \ar[r]  \ar[d]_{\ch}^{(2)} &
K_1(Y_\bull)_\Q \ar[r]   \ar[d]_{\ch}^{(3)} &  \\
\bigoplus\limits_{r\le d+2} H^{2r-2} ( \Z_{Y_1}(r))_\Q  \ar[r] &
\bigoplus\limits_{r\le d+2} H_{\cont}^{2r-2}( p(r) \Omega_{Y_\bull}^{<r} )_\Q \ar[r]  & \bigoplus\limits_{r\le d+2} H_{\cont}^{2r-1}( \Z_{Y.}(r) )_\Q   \ar[r]  & 
 }\\
&\xymatrix{ \ar[r]  &
K_1(Y_1)_\Q \ar[r]  \ar[d]_{\ch}^{(4)}  \ar[r] &  K_0(Y_\bull, Y_1 )_\Q  \ar[d]_{\ch}^{(5)}\\
 \ar[r] &   \bigoplus\limits_{r\le d+2} H^{2r-1}( \Z_{Y_1}(r) )_\Q  \ar[r]& \bigoplus\limits_{r\le d+2} H_{\cont}^{2r-1}( p(r) \Omega_{Y_\bull}^{<r} )_\Q
} \notag
\end{align}
where the lower row comes from the fundamental triangle, Proposition~\ref{diag}.
In order to show that $(3)$ is an isomorphism it suffices to observe:
\begin{itemize}
\item[(a)] \hspace{5mm} the map  $(1)$ is surjective and $(4)$ is bijective,
\item[(b)] \hspace{5mm} the map $(2)$ is bijective and the map  $(5)$ is  injective.
\end{itemize}
Part (a) is shown in \cite[Thm.~9.1]{Bl}. We show part (b).

{From} Proposition~\ref{topcyc.desc}(b) and Lemma~\ref{hoto.descentBK} we get a convergent \'etale descent spectral sequence of Bousfield-Kan type
\begin{equation}\label{ci.spse1}
E_2^{s,t}(K) = H^s_{\cont}(Y_{1,\et} , \sK_{Y_\bull,Y_1,t} )  \Longrightarrow K^\cont_{t-s}(Y_\bull, Y_1)
\end{equation}
As coherent sheaves satisfy \'etale descent we also get from Lemma~\ref{hoto.deschc} a
spectral sequence with Bousfield-Kan type renumbering
\begin{equation}\label{ci.spse2}
E_2^{s,t}(\Z(r)) = H^s_{{\cont}}(Y_{1,\et} , \sH^{2r-t-1}(p(r) \Omega^{<r}_{Y_\bull}  ) )
\Longrightarrow H^{2r-t+s-1}_\cont(Y_1,p(r) \Omega^{<r}_{Y_\bull} ).
\end{equation}
The  Chern character on relative theories induces a morphism of spectral sequences from \eqref{ci.spse1} to
\eqref{ci.spse2}. 
Note that $E_2^{s,t}(K)=E_2^{s,t}(\Z(r))=0$ if $s>d+2$, because ${\rm cd}_p (Y_1)\le
d+1$ \cite[Thm~5.1, Exp.~X]{SGA4} and the relative $K$-sheaves are $p$-primary torsion by Proposition~\ref{topcyc.desc}(a).

By Lemma~\ref{hoto.desciso} in order to show (b) it is enough to show that the Chern character induces an
isomorphism
\[
{\rm ch}: E_2^{s,t}(K) \to \bigoplus_{r\le d+2} E_2^{s,t}(\Z(r))
\]
for $0\le t-s\le 2$ and $s\le d+2$. This follows from:

\begin{claim}\label{claim.locmain}
The Chern character induces an isomorphism of \'etale pro-sheaves
\[
\ch: \sK_{Y_\bull,Y_1,a} \to \bigoplus_{r\le a} \sH^{2r-a-1} ( p(r)\Omega^{<r}_{Y_\bull}) 
\]
for $1\le a<p-2$.
\end{claim}

{\em Case $a=1$}: It is known that $\sK_{Y_1,2}$ is locally generated by Steinberg symbols
\cite{DS}, so $\sK_{Y_\bull , 2} \to \sK_{Y_1,2}$ is surjective and therefore
$\sK_{Y_\bull, Y_1, 1}= (\G_m)_{Y_\bull,Y_1}$. The target set of the Chern character for
$a=1$ is just $p \sO_{X_\bull}$ and the Chern character is the $p$-adic logarithm
isomorphism in this case because of the isomorphism in Proposition~\ref{mot.basicprop}(1).

\medskip

{\em Case $a>1$}: By Proposition~\ref{topcyc.desc}(a) there is an isomorphism of
pro-sheaves
\[
\sK_{Y_\bull,Y_1,a} \xrightarrow{\sim}   (\sK /p^\bull )_{Y_\bull,Y_1,a}
\] 
and similarly for relative motivic cohomology. By a simple d\'evissage it therefore suffices
to show that the Chern character of \'etale pro-sheaves
\[
\ch: (\sK / p )_{Y_\bull,Y_1,a} \to \bigoplus_{r\le a} \sH^{2r-a-1} (
p(r)\Omega^{<r}_{Y_\bull} \otimes_\Z \Z/p ) 
\]
is an epimorphism for $2\le a< p-1$ and a monomorphism for $2\le a< p-2$.  

Observe that
\begin{equation}\label{ci.iso2}
\ch: (\sK / p)_{Y_1,a} \to \sH^{a}(\Z_{Y_1}(a) \otimes_\Z \Z/p ) \\
\end{equation}
is an isomorphism for all $a<p$.
Concerning \eqref{ci.iso2}, note that $ \sH^{a}(\Z_{Y_1}(r) \otimes_\Z \Z/p )=0$ for $r\ne
a$ by \cite{GL}.
Indeed, Geisser-Levine show that there is precisely one such morphism \eqref{ci.iso2} compatible with
Steinberg symbols
on both sides, which our Chern character is, and that this one morphism is an
isomorphism.

Using the sheaf analog of the commutative diagram of exact sequences \eqref{ci.relseq},
the isomorphism \eqref{ci.iso2} and the following claim, we finish the proof of
Theorem~\ref{ciso.isothm}.

\begin{claim}
The Chern character induces an isomorphism
\begin{equation} \label{ci.iso1} 
\ch :   (\sK / p)_{Y_\bull,a} \to   \bigoplus_{r\le a} \sH^{2r-a} (
\Z_{Y_\bull}(r) \otimes_\Z \Z/p ) 
\end{equation}
 for $2\le a<p-1$.
\end{claim}

In order to prove the claim we can assume that $Y_\bull$ is affine. Then by
\cite[Thm.~7]{Elk} our $Y_\bull$ is the $p$-adic formal scheme associated to a smooth affine
scheme $Y/W$. 
With the notation as in Section~\ref{topcyc}, in particular with $i:Y_1 \to Y$ the
immersion of the closed fibre, there is a commutative diagram
\[
\xymatrix{
i^* (\sK / p)_{Y,a} \ar[r]^-{\ch} \ar[d]_{\wr} &   \bigoplus_{r\le a}i^* \sH^{2r-a} (
\frak V_Y(r) )   \ar[d]_\wr \\
(\sK / p)_{Y_\bull,a} \ar[r]^-{\ch} &   \bigoplus_{r\le a} \sH^{2r-a} (
\Z_{Y_\bull}(r) \otimes_\Z \Z/p ) 
}
\]
The right vertical isomorphism is due to Kurihara~\cite{Ku1} and the left vertical
isomorphism is from Proposition~\ref{topcyc.cont}.
The top horizontal map is
induced by Sato's Chern character~\cite[Sec.~4]{Sa}. The square commutes, because Sato's
Chern character is also constructed in terms of universal Chern classes analogous to our construction in Section~\ref{chern}.
 
In order to show that our Chern character induces an isomorphism as the lower horizontal
map in the commutative square we can make the base change 
$W \subset W[\zeta_p]$ with $\zeta_p$ a primitive $p$-th root of unity. Then it is clear
that Sato's Chern character maps the Bott element to the Bott element and is compatible
with Steinberg symbols. Therefore Proposition~\ref{topcyc.isoprop} shows that the top
horizontal map is an isomorphism. 
 \end{proof}

\medskip

In order to finish the proof of the Main Theorem~\ref{MainThm}, combine
Theorem~\ref{def.thm} with Theorem~\ref{ciso.isothm}.

\medskip

As a direct generalization of Theorem~\ref{ciso.isothm} we obtain

\begin{thm} \label{higherK}
For $i>0$ and $p>d+i+5$ the Chern character
\[
\ch : K^\cont_i(X_\bull)_\Q  \to \bigoplus_{r\le d+i} H^{2r-i}_\cont(X_1 , \Z(r)_{X_\bull})_\Q
\]
is an isomorphism.
\end{thm}

In fact in the previous proof one omits the delooping trick at the beginning and then
reduces in the same way to Claim~\ref{claim.locmain}.

\section{Milnor $K$-theory} \label{localcomputations}

\noindent 
In this section we recall some properties of Milnor $K$-theory and we study the infinitesimal part of Milnor $K$-groups
for smooth rings over $W_n$, recollecting  results of Kurihara \cite{Ku}, \cite{Ku2}. The
main result of this section, Theorem~\ref{thm2.1}, is used in
Proposition~\ref{mot.basicprop}(4) to relate Milnor $K$-theory and motivic cohomology of a
$p$-adic scheme.

Consider the functor 
\[
F:A \mapsto \otimes_{n\ge 0} (A^\times )^{\otimes n} /{\it St}
\]
from commutative rings to graded rings, where ${\it St}$ is the graded two-sided ideal generated by elements $a
\otimes b$ with $a + b =1$.  

Let $S$ be a base scheme and let
$F^\sim$ be the sheaf on the category of schemes over $S$  associated to the functor $F$ in either the Zariski,
Nisnevich or \'etale topology.
The Milnor $K$-sheaf $\mathcal K^M_*$ is a certain quotient sheaf of $F^\sim$, defined in
\cite{Ke2}. 
In particular it is locally generated by symbols 
\[
\{ x_1,\ldots , x_r \}\;\;\;\; \text{ with }\; x_1, \ldots , x_r \in \mathcal O^\times .
\] 

In fact, if the residue fields at all points of $S$ are infinite, the map
$F^\sim \to \mathcal K^M_*$ is an isomorphism. For a scheme $X/S$ denote by
$\mathcal K^M_{X,*}$ the restriction of $\mathcal K^M_*$ to the small site of $X$.

\medskip

Let $S=\Spec \,k$ for a perfect field $k$ with $\Char\, k = p >0$ and let $X\in \Sm_k$.
\begin{prop}\label{mil.BKiso}\mbox{}
\begin{itemize}
\item[(a)]
The sheaf
$\mathcal K^M_{X,*}$ is $p$-torsion free. 
\item[(b)] 
The composite of the Teichm\"uller lift and the $d \log$-map induces an
isomorphism 
\[
 d \log \, [ - ]: \mathcal K_{X,r}^M / p^n  \xrightarrow{\simeq }   W_n \Omega^r_{X,\log} 
\]
with the logarithmic de Rham-Witt sheaf.
\end{itemize}
\end{prop}

\begin{proof}
Part (a) is due to Izhboldin~\cite{Iz}. Part (b) is due to Bloch-Kato~\cite{BK}.
\end{proof}

\medskip

Let $R$ be
an essentially smooth local ring over $W_n=W(k)/ p^n$. By $R_1$ we denote
$R/(p)$. In this section, we study Milnor $K$-groups of $R$.

By the Milnor $K$-group $K^M_r(R)$ we mean the stalk of the Milnor $K$-sheaf in Zariski
topology over $\Spec R$.
We consider the filtration $U^i K^M_r (R) \subset K^M_r(R)$ ($i\ge 1$), where
$U^i K^M_r (R)$ is generated by symbols
\[
\{ 1 + p^i x, x_2, \ldots , x_r \}
\]
with $x\in R $ and $x_i \in R^\times$ ($2\le i \le r$).
One easily shows that $U^1K^M_r(R)  $ is equal to the kernel of
$K^M_r(R) \to K^M_r(R_1) $.

\begin{lem}\label{mil.torsion}
The group $U^1 K^M_r(R) $ is $p$-primary torsion of finite exponent.
\end{lem}
\begin{proof}
Without loss of generality we can assume $r=2$.
The theory of pointy bracket symbols for the relative $K$-group
$K_2(R,pR)$  (\cite{SK}), yields generators $\pb{a,b}$ of $U^1 K^M_r(R)$
defined for
 $a,b \in R$
with at least one of $a,b \in pR$. Relations for the pointy brackets
are:\nn
(i) $\pb{a,b} = -\pb{b,a};\ a\in R, b\in pR \text{ or }b\in R, a\in
pR$\nn (ii) $\pb{a,b}+\pb{a,c} = \pb{a,b+c-abc};\ \ a\in pR \text{ or
}b,c\in pR$\nn (iii) $\pb{a,bc}=\pb{ab,c}+\pb{ac,b};\ a\in pR$.

Note that for $a$ fixed, the mapping $(b,c) \mapsto b+c-abc$ is a formal group
law. It follows that for $N\gg 0$, $p^N\pb{a,b}= \pb{a,0}=0$, so $K_2(R,pR)$
is $p$-primary torsion of finite exponent.
\end{proof}

\begin{thm}\label{thm2.1} For $p>2$ the assignment 
\eq{}{\label{def.Exp} pxd\log y_1\wedge\ldots\wedge d\log y_{r-1} \mapsto
\{\exp(px),y_1,\dotsc,y_{r-1}\}
} 
induces an isomorphism
\eq{}{\Exp:  p\Omega^{r-1}_{R_n}/p^2d\Omega^{r-2}_{R_n} \stackrel{\sim}{\to} U^{1}K^M_r(R_n). 
}
\end{thm}

\begin{proof}
\mbox{}
\medskip

{\em 1st step: } $\Exp: p \Omega^{r-1}_R \to K^M_r(R)$  as in
\eqref{def.Exp}
is well-defined.

\medskip

Note that Kurihara \cite{Ku2} shows the exponential map is well defined if
$K^M_r(R)$ is replaced by its $p$-adic completion
$K^M_r(R)^\wedge_p$. 
By standard arguments, see \cite[Sec.\ 3.1]{Ku2}, we reduce to $r=2$.
By Proposition~\ref{mil.BKiso}(a) the group  $K^M_2(R_1)$ has no $p$-torsion. This implies
that for any $n\ge 1$
\begin{equation}\label{mil.modpexact}
0 \to U^1 K^M_2(R) \otimes \Z/ p^n  \to  K^M_2(R) \otimes \Z/ p^n \to  K^M_2(R_1) \otimes \Z/ p^n
\to 0
\end{equation}
is exact. For $n\gg 0$ Lemma~\ref{mil.torsion} says that $U^1 K^M_2(R)
\otimes \Z/p^n = U^1 K^M_2(R)$. Taking the inverse limit over $n$ in
\eqref{mil.modpexact} we see that 
\begin{equation}\label{mil.injeq}
 U^1 K^M_2(R)   \to
K^M_2(R)^\wedge_p
\end{equation}
 is injective. So the claim follows from the result
of Kurihara mentioned above.

\medskip

{\em 2nd step: } $\Exp(p^2 d \Omega^{r-2}_R ) = 0$

\medskip

Without loss of generality $r=2$. The claim follows from the
injectivity of \eqref{mil.injeq} and \cite[Cor.\ 1.3]{Ku2}.

\medskip

{\em 3rd step: } $\Exp:  p\Omega^{r-1}_{R_n}/p^2d\Omega^{r-2}_{R_n}
{\to} U^{1}K^M_r(R_n)  $ is an isomorphism.

\medskip

Set $G_r = p \Omega^{r-1}_R / p^2 d \Omega^{r-2}_R$ and define a
filtration on it by the subgroups $U^i G_r \subset G_r $ ($i\ge 1$) given by
the images of $p^i \Omega^{r-1}_R$. Note that
\[
\gr^i G_r =
\Omega^{r-1}_{R_1} / B_{i-1} \Omega^{r-1}_{R_1} ,
\]
see \cite[Cor.\ 0.2.3.13]{Il}. In \cite[Prop.\ 2.3]{Ku} Kurihara shows
that 
\[
\gr^i G_r \to \gr^i K^M_r(R)
\]
is an isomorphism.
This finishes the proof of the theorem.
\end{proof}

\begin{appendix} \section{Homological algebra} \label{sec.hom}

\noindent
In this section we collect some standard facts from homological algebra that we use.
Let $\mathcal T$ be a triangulated category with $t$-structure, see
\cite[Sec.~1.3]{BBD}. 

\begin{lem}\label{homalg.lem1}
For an integer $r$ and for an exact triangle
\[
A\to B \to C \xrightarrow{[1]} A[1]
\]
in $\mathcal T$
with $A \in \mathcal T^{\le r}$  the triangle
\[
A \to \tau_{\le r} B \to \tau_{\le r} C \xrightarrow{[1]} A[1]
\]
is exact.
\end{lem}

\begin{lem}\label{hom.unicity}
For $A,B \in \mathcal T$ with $A\in \mathcal T^{\le r}$ and $B\in \mathcal T^{\le r} \cap
\mathcal T^{\ge r} $ assume given an epimorphism $\mathcal H^r (A) \to \mathcal H^r(B)$.
Then this epimorphism lifts uniquely to a morphism
$A\to B$ in $\mathcal T$, sitting inside an exact triangle
\[
A \to B \to C \to A[1]
\]
which is unique up to unique isomorphism.
\end{lem}

\begin{proof}
The existence of such an exact triangle is clear from the axioms of triangulated
categories. Note that $C\in \mathcal T^{<r} $. Uniqueness means that there exists a unique
dotted isomorphism $\alpha$ in a
commutative diagram with exact triangles as rows
\[
\xymatrix{
A \ar@{=}[d] \ar[r] &  B \ar@{=}[d] \ar[r] &  C \ar@{..>}[d]^{\alpha} \ar[r]  &  A[1] \ar@{=}[d] \\
 A \ar[r] &  B \ar[r] & C' \ar[r] &  A[1]
}
\]
Existence and uniqueness follow from the exact sequence
\[
0= \Hom(C,B) \to \Hom(C,C') \to \Hom(C, A[1]) \to \Hom(C, B[1]) .
\]
\end{proof}

Now we discuss pro-sheaves on sites.
Let $\N$ be the category with the objects $\{1,2,3 , \ldots \}$ and morphisms $n_1 \to n_2$
for $n_1\ge n_2$. 
By the category of {\it pro-systems} ${\rm C}_\pro$, for a category $\C$, we mean the category of diagrams
in $\rm C$ with index category $\N$ and with morphisms $${\rm Mor}_{{\rm C}_\pro} (Y_\bull, Z_\bull) = \varprojlim_{n}
\varinjlim_{m} {\rm Mor}_{\rm C}(Y_m, Z_n).$$
\begin{defn} Let $\bS$ be a small site.
\begin{itemize}
\item[(a)] By $\Sh(\bS)$ we denote the category of sheaves of abelian groups on $\bS$.
 By $\C(\bS)$ we denote the category of unbounded complexes in $\Sh(\bS)$.
\item[(b)]
 By $\Sh_\pro(\mathbb S)$ we denote the category of pro-systems in $\Sh(\bS)$.
   \item[(c)]
By $\C_\pro ( \bS)$ we denote
  the category of pro-systems in $\C(\bS)$.
\item[(d)] By $\D_\pro(\mathbb S)$ we denote the Verdier localization of the homotopy
  category of $\C_\pro(\bS)$, where we kill objects which are represented by systems of complexes which have
  level-wise vanishing cohomology sheaves.  

\end{itemize}
\end{defn}

For the construction of Verdier localization in (d) see \cite[Sec.~2.1]{Ne}.

\begin{lem}\mbox{}
The triangulated category $\D_\pro(\mathbb S)$ has a natural $t$-structure $(D^{\le 0}(\bS),
D^{\ge 0}(\bS))$ with
$\sF_\bull\in \D_\pro^{\le 0}$ resp.\ $\sF_\bull\in \D_\pro^{\ge 0}$ if $\sF_\bull$ is
isomorphic in $\D_\pro(\mathbb S)$ to $\sF'_\bull$ with $\mathcal H^i(\sF'_n)=0$ for all $n\in \N$ and $i> 0$ resp.\ for $i<0$. 
 The $t$-structure has heart 
$\Sh_\pro(\mathbb S)$.
\end{lem}

We write $D^+_\pro(\bS)$, $D^-_\pro(\bS)$ and $D^b_\pro(\bS)$ for the bounded above, bounded below and
bounded objects in $D(\bS)$ with respect to the $t$-structure.

\section{Homotopical algebra}\label{sec.hoto}

\noindent
In this section we introduce certain standard model categories of pro-systems 
over a small site $\mathbb S$.
We uniquely specify our model structures by explaining what are the cofibrations and weak
equivalences. The fibrations are then defined to be the maps which have the right
lifting property with respect to all trivial cofibrations. Our definition of closed model
category is as in~\cite{Q}. 

\begin{defn}\mbox{}
\begin{itemize}
\item[(a)] Let $\S(\bS)$ be the proper closed simplicial model category of simplicial
  presheaves on $\bS$, where cofibrations are injective morphisms of presheaves and weak
  equivalences are those maps which induce isomorphisms on homotopy sheaves,
  cf.~\cite[Sec.~2]{Jar1}. 
\item[(b)]
We endow the category of unbounded complexes of abelian sheaves $\C(\bS)$ with the proper closed simplicial
model structure where cofibrations are injective morphisms and weak equivalences are those
maps which induce isomorphisms on cohomology sheaves, see~ App.\ C in \cite{CTHK} and Thm.~2.3.13 in \cite{Hov}.
\end{itemize}
\end{defn}

Explicit characterizations of the classes of fibrations for the two model categories are given in the
references. 
 For the crucial notion  of level representation in the
following definitions see \cite[Sec. 2.1]{Isa1}.

\begin{defn}\label{hoto.defnpro}  \mbox{}
\begin{itemize}
\item[(a)]
By $\S_\pro(\bS)$ we denote the proper closed simplicial model category of pro-systems of simplicial
presheaves on $\bS$,
where cofibrations are those maps which have a level representation by levelwise injective
morphisms and where weak equivalences are those maps which have a level representation
which induces a levelwise isomorphism on homotopy sheaves.  
\item[(b)] 
We endow $\C_\pro(\bS)$ with the proper closed simplicial model structure,
where cofibrations are those maps which have a level representation by levelwise injective
morphisms and where weak equivalences are those maps which have a level representation
which induces a levelwise isomorphism on cohomology sheaves. 
\end{itemize}
\end{defn}

\begin{nota}\label{hoto.homcat}
For a model category $\rm M$ we write $\rm hM$ for the associated homotopy category. 
\end{nota}

The pro-model structures in Definition~\ref{hoto.defnpro} are due to Isaksen \cite{Isa1}.
He uses all pro-systems indexed by small cofiltering categories, whereas we allow only
$\N$ as index category. In fact all his definitions and proofs work in a simpler
way in this setting, except for the following points:
 In our model categories only countable inverse 
limits and finite direct limits exist, cf.\ \cite[Sec.\ 11]{Isa2}. Also for our categories
the simplicial functors $K\otimes -$ resp.\ $(-)^K$ exist only for a finite resp.\ countable simplicial set
$K$. This is why we use Quillen's original notion of a closed
simplicial model category \cite{Q}.
 Note that Isaksen calls his pro-category strict model category.

Isaksen gives the following concrete description of fibrations.

\begin{prop}\label{hoto.expfib}
(Trivial) fibrations in $\S_\pro(\bS)$ resp.\ $\C_\pro(\bS)$ are precisely those maps, which are
retracts of maps having a level representation $f:X_\bull \to Y_\bull $ such that
\[
f_n:X_n \to X_{n-1} \times_{Y_{n-1}} Y_n 
\] 
are (trivial) fibrations in $\S (\bS)$ resp.\ $\C(\bS)$ for $n\ge 1$. Here we let $X_0=Y_0$  be the
final object.
\end{prop}

\smallskip

{\em Sketch of Isaksen's construction (Definition~\ref{hoto.defnpro})}.
In a first step one shows the two out of three property for weak equivalences. The key
lemma in this step is \cite[Lem.\ 3.2]{Isa1}, which is the only part of the construction
where Isaksen constructs a new non-trivial index category. For index category $\N$ the
argument simplifies. In a second step one shows the various left and right lifting
properties of a model category. Here one takes the description of fibrations given in
Proposition~\ref{hoto.expfib} as a definition and thereby also obtains a proof of this proposition.
\medskip

\begin{prop}\label{hoto.adj}\mbox{}
\begin{itemize}
\item[(a)]
There are Quillen adjoint functors 
\[
\xymatrix{
 \S_\pro(\bS)   \ar@<2pt>[r]  &  \C_\pro(\bS)  \ar@<2pt>[l]^{\rm K}
} 
\]
where the right adjoint $\rm K$ is the composition of the good truncation $\tau_{\le 0}$ and the
Eilenberg-MacLane space construction.
\item[(b)]
There is a canonical ismorphism of categories 
\begin{align*}
\D_\pro(\bS) &\xrightarrow{\simeq} {\rm hC}_\pro(\bS) 
\end{align*}
\item[(c)] There are Quillen adjoint functors
\[
\xymatrix{
\S (\bS)   \ar@<2pt>[r]  &  \S_\pro(\bS)  \ar@<2pt>[l]^-{\varprojlim} ,\\
\C (\bS)   \ar@<2pt>[r]  &  \C_\pro(\bS)  \ar@<2pt>[l]^-{\varprojlim},
} 
\] 
where the left adjoint is the constant pro-system functor and the right adjoint is the
inverse limit functor.
\end{itemize}
\end{prop}

\begin{nota} \mbox{}
\begin{itemize}
\item
We write
\[
{\rm K}: {\rm hC}_\pro(\bS) \to {\rm hS}_\pro (\bS)
\]
for the functor induced by $\rm K: \C_\pro (\bS ) \to \S_\pro( \bS)$.
\item
We write $[Y_1,Y_2]$ for the set of morphisms from $Y_1$ to $Y_2$ in the homotopy category.
\item
 The right derived functor ${\rm holim}: {\rm hS}_\pro(\bS) \to {\rm hS}_\pro(\bS ) $ of
 $\varprojlim: \S_\pro(\bS )\to \S (\bS)$ is called homotopy inverse
 limit. By $R \varprojlim : \D_\pro(\bS) \to \D(\bS)$ we denote the right derived functor
 of $\varprojlim: \C_\pro(\bS )\to \C (\bS)$.
\end{itemize} 
\end{nota}

There is a standard method for calculating the derived inverse limit $R^i \varprojlim:
\Sh_\pro (\bS) \to \Sh(\bS)$ which shows in particular that $R^i \varprojlim =0$ for
$i>1$, see \cite[Sec. 3.5]{Weib}.

\begin{defn}\label{hoto.contcoho}
We define continuous cohomology of $\sF_\bull\in \D_\pro(\bS)$ by
\[
H^i_\cont(\bS , \sF_\bull ) = [\Z[-i], \sF_\bull  ] ,
\]
where $\Z$ denotes the constant sheaf of integers.
\end{defn}
Continuous cohomology of sheaves was first studied in \cite{Ja}. 
Note that we have a short exact sequence
\begin{equation}\label{hom.lim1}
0\to \varprojlim_n{\!}^1 H^{i-1}(\mathbb S , \sF_n )  \to H_\cont^i(\mathbb S, \sF_\bull)
\to \varprojlim_n H^i(\mathbb S, \sF_n ) \to 0.
\end{equation}

\begin{lem}\label{hoto.deschc}
For $\sF_\bull\in \D_\pro^+(\bS)$ there is a convergent
 spectral sequence
\[
E_2^{p,q} = H^p_\cont (\bS, \sH^q(\sF_\bull)) \Longrightarrow H^{p+q}_\cont (\bS , \sF_\bull)
\]
with differential $d_r:E_r^{p,q} \to E_r^{p+r,q - r+1}$.
\end{lem}

\begin{lem}\label{hoto.descentBK} 
Let $C_\bull$ be a pointed object in $ \S_\pro(\bS)$ and assume that  $\tilde\pi_1(\C_n)$ is commutative for any
 $n\ge 1$. If there is $N$ such that $H^i_\cont (\bS
, \tilde\pi_j(C_\bull)) =0$ for $i>N$ and $j>0$, then there is a completely convergent  Bousfield-Kan spectral sequence
\[
E^{s,t}_2 = H^{s}_\cont (\bS, \tilde{\pi}_t (C_\bull) ) \Longrightarrow [S^{t-s}, C_\bull]
\quad \quad \text{ with }\quad t\ge s
\]
and differential $d_r:E_r^{s,t} \to E_r^{s+r,t+r-1}$.
\end{lem}

Here $\tilde\pi_i$ is the pro-system of sheaves of homotopy groups and $H^0_\cont$ of the
sheaf of sets $\tilde \pi_0(C_\bull)$ means simply global sections of the inverse limit. The indexing of the
spectral sequence is as in \cite[Sec.~IX.4.2]{BoK}.

For $C_\bull = {\rm K} (\sF_\bull)$ with $\rm K$ as in Proposition~\ref{hoto.adj}(a) and
$\sF_\bull$ as in Lemma~\ref{hoto.deschc} there is a natural 
morphism \[E^{s,t}_r(\sF_\bull)\to E^{s,t}_r( {\rm
  K}(\sF_\bull) ) \quad \quad  (t\ge s, r\ge 2),  \]
compatible with the differential $d_r$, 
 where the left side is a Bousfield-Kan renumbering of the
spectral sequence of Lemma~\ref{hoto.descentBK} and the right side is the spectral
sequence of Lemma~\ref{hoto.descentBK}. This morphism is injective for $t=s$ and bijective
for $t>s$. 

Lemma~\ref{hoto.descentBK} implies in particular the following lemma.

\begin{lem}\label{hoto.desciso}
Let $C_\bull, C'_\bull\in \S_\pro(\bS) $ satisfy the assumptions of
Lemma~\ref{hoto.descentBK} and let $\Psi: C_\bull \to C'_\bull$ be a morphism.
\begin{itemize}
\item[(a)]
 Assume that for an integer $n\ge 1$ the induced map 
\begin{equation}\label{hoto.eqshm}
 H^{s}_\cont (\bS, \tilde{\pi}_t (C_\bull) )  \xrightarrow{\Psi_*}  H^{s}_\cont (\bS, \tilde{\pi}_t (C'_\bull) ), 
\end{equation}
 is injective for all $t,s$ with $t-s=n-1$, bijective for
$t-s=n$ and surjective for $t-s=n+1$. Then $\Psi_*:[S^{n}, C_\bull]\to [S^{n}, C'_\bull]$
is an isomorphism.
\item[(b)]
Assume that \eqref{hoto.eqshm} is surjective for $t-s=1$ and injective for $t=s$. Then
 $\Psi_*:[S^{0}, C_\bull]\to [S^{0}, C'_\bull]$ is injective. 
\end{itemize}
\end{lem}

\begin{defn}\label{hoto.descent}
An object $C_\bull\in \S_\pro (\bS)$ satisfies descent if for any object $U\in \bS$
\[
\Gamma (U,C_\bull) \to \Gamma(U, {\rm F}C_\bull)
\]
is a an isomorphism in ${\rm hS}_\pro (\{*\})$. Here ${\rm F}C_\bull$ is a fibrant
replacement in $\S_\pro(\bS)$.
\end{defn}

\section{The Motivic Complex: a Crystalline Construction,}\label{sec.cryscon}

\noindent
In this appendix we continue to assume that $r<p$. We identify the motivic complex $\widetilde\Z_{X_\bull}(r)$ as constructed in Section~\ref{sec.cand} with the complex $\Z_{X_\bull}(r)$ given in definition \ref{Z(r)}.   The later is defined via a cone involving the Nisnevich syntomic complex $\frak{S}_{X_\bull}(r)$  (Definition~\ref{S}). As a preliminary simplification, we may modify the cone \eqref{3.3a} and define 
\ml{1a}{\widetilde{\frak{S}}_{X_\bull}(r) := {\rm Cone}\Big(I(r)\Omega^\bulle_{D_\bull}\oplus \Omega^{\ge r}_{X_\bull}\oplus W\Omega^\bulle_{X_1,\log}[-r]\xrightarrow{\psi} \\p(r)\Omega^\bulle_{X_\bull}\oplus q(r)W\Omega^\bulle_{X_1}\Big).
}
Here $\psi_{2,3}: \Omega^\bulle_{X_1,\log}[-r] \to q(r)W\Omega^\bulle_{X_1}$ is the natural inclusion. We will exhibit a canonical quasi-isomorphism $\widetilde{\frak{S}}_{X_\bull}(r) \simeq \frak{S}_{X_\bull}(r)$.  The desired result for motivic cohomology will follow by a further cone construction for the map $d\log: \Z_{X_1}(r) \to W\Omega^\bulle_{X_1,\log}[-r]$ \eqref{7.4a}. 

Let us write $C^\bulle := \widetilde{\frak{S}}_{X_\bull}(r)$. 
\begin{lem}\label{lem1} $\sH^j(C^\bulle) = (0)$ for $j\ge r+1$, i.e. $\tau_{\le r}C^\bulle \xrightarrow{\simeq} C^\bulle$. 
\end{lem}
\begin{proof} It suffices to show the map
\eq{}{\sH^j\Big(I(r)\Omega^\bulle_D \oplus W\Omega^r_{X_1,\log}[-r]\Big) \to \sH^j\Big(p(r)\Omega^{\le r-1}_{X_\bull} \oplus q(r)W\Omega^\bulle_{X_1}\Big)
}
is an isomorphism for $j\ge r+1$ and is surjective for $j=r$. This follows from the assertion $I(r)\Omega^\bulle_D \simeq q(r)W\Omega^\bulle_{X_1}$ which is a consequence of formulas \eqref{2.4a} and \eqref{2.5a} in the paper. \end{proof}

\begin{lem}\label{lem2} Let $\ve: X_{\et} \to X_{\Nis}$ be the map of sites. Then there is a canonical quasi-isomorphism $\ve^*C^\bulle\simeq \frak{S}_{X_\bull}(r)_{\et}$,  the syntomic complex in the \'etale topology. 
\end{lem}
\begin{proof}There is a natural inclusion of cones 
$${\rm Cone}(W\Omega^r_{X_1,\log}[-r] \to q(r)W\Omega^\bulle_{X_1})[-1] \to C^\bulle. 
$$
In the \'etale site, the cone on the left is quasi-isomorphic to $W\Omega^\bulle_{X_1}[-1]$ (Corollary \ref{syn.cor1}). As a consequence, in the \'etale site we get
\eq{}{C^\bulle[-1] \simeq {\rm Cone}(\Omega_{X_\bull}^{\ge r} \oplus p(r)\Omega^\bulle_D \to p(r)\Omega^\bulle_{X_\bull}\oplus W\Omega^\bulle_{X_1})[-1]. 
}
(The map $p(r)\Omega^\bulle_D \to W\Omega^\bulle_{X_1}$ is $(1-F_r)\circ \mu$, where $\mu: \Omega^\bulle_D \to W\Omega^\bulle_{X_1}$ is the composition of \eqref{2.4a} and \eqref{2.7a} in the paper.) Let $\xi: J(r)\Omega^\bulle_{D_\bull} \to \Omega_{X_\bull}^{\ge r}$ be as in \eqref{2.8a} in the paper, and let $\iota: J(r)\Omega^\bulle_{D_\bull} \subset \Omega^\bulle_{D_\bull}$ be the natural inclusion. Construct a commutative diagram
\eq{4a}{\begin{CD} J(r)\Omega^\bulle_{D_\bull} @>1-f_r>> \Omega^\bulle_{D_\bull} \\
@VV(-\xi, \iota)V @VV(0, \mu) V \\
\Omega_{X_\bull}^{\ge r} \oplus I(r)\Omega^\bulle_{D_\bull} @>>> p(r)\Omega^\bulle_{X_\bull}\oplus W\Omega^\bulle_{X_1}.
\end{CD}
}
This diagram yields the desired quasi-isomorphism in the \'etale site. 
\end{proof}

We have by Lemma \ref{lem2}, $C^\bulle \to R\ve_*\ve^*C^\bulle \simeq R\ve_*\frak{S}_{X_\bull}(r)_{et}$. Applying $\tau_{\le r}$ and using Lemma \ref{lem1} we get 
\eq{5}{C^\bulle \simeq \tau_{\le r} C^\bulle \to \tau_{\le r}R\ve_*\frak{S}_{X_\bull}(r)_{et}=:\frak{S}_{X_\bull}(r)_{Nis}. 
}

We must show the map \eqref{5} is a quasi-isomorphism. Consider the commutative diagram
\eq{6a}{\begin{CD}\Omega_{X_\bull}^{\ge r} \oplus I(r)\Omega^\bulle_{D_\bull} @>>> p(r)\Omega^\bulle_{X_\bull}\oplus W\Omega^\bulle_{X_1} \\
@AAA @AAA \\
\Omega_{X_\bull}^{\ge r} \oplus I(r)\Omega_D \oplus W\Omega_{X_1,\log}^r[-r] @>>> p(r)\Omega^\bulle_{X_\bull}\oplus q(r)W\Omega^\bulle_{X_1}
\end{CD}
}
Here the bottom line is as in \eqref{1a} and the top as in \eqref{4a}. The sheaves on the top are $\ve$-acyclic, so the top complex represents $R\ve_*\frak{S}_{X_\bull}(r)_{\et}$ and the whole diagram represents $C^\bulle \to R\ve_*\ve^*C^\bulle \simeq R\ve_*\frak{S}_{X_\bull}(r)_{\et}$. It will suffice to check that this vertical map of Nisnevich complexes induces an isomorphism in cohomology in degrees $\le r$. 

In the Nisnevich topology, consider the double complex of complexes which we position so $W\Omega^r_{X_1,\log}[-r]$ is in position $(0,0)$. 
\eq{7a}{\begin{CD}0 @>>> W\Omega^\bulle_{X_1} \\
@AA 0 A @AA 1-F_r A \\
W\Omega_{X_1,\log}^r[-r] @>\inj >> q(r)W\Omega^\bulle_{X_1}
\end{CD}
}
\begin{lem}The total complex of Nisnevich sheaves associated to \eqref{7a} is acyclic away from degree $r+2$.
\end{lem}
\begin{proof}Writing $T$ fo the total complex, we have a triangle 
$$T \to (q(r)W\Omega^\bulle/W\Omega^r_{\log}[-r])[-1] \xrightarrow{1-F_r} W\Omega^\bulle[-1] \xrightarrow{+1}
$$
The map $1-F_r$ induces isomorphisms in cohomology (Lemmas 3.4 and 3.5) except 
$$1-F_r : \sH^r\Big(q(r)W\Omega^\bulle/W\Omega^r_{\log}[-r]\Big) \cong  \sH^r(q(r)W\Omega^\bulle)/W\Omega^r_{\log} \to \sH^r(W\Omega^\bulle)
$$
is not surjective so $\sH^i(T) = (0), i\neq r+2$. \end{proof}

Let $U$ be the corresponding total complex for the diagram \eqref{6a}. The inclusion $T \inj U$ is a quasi-isomorphism so $\sH^i(U) = (0), i\neq r+2$. It follows that $\sH^i(C^\bulle) \to \sH^i(R\ve_*\frak{S}_{X_\bull}(r)_{et})$ is an isomorphism except possibly for $i=r+1, r+2$. This implies $\sH^i(\tau_{\le r}C^\bulle) \to \sH^i(\tau_{\le r}R\ve_*\frak{S}_{X_\bull}(r)_{\et})$ is a quasi-isomorphism for all $i$. From Lemma \ref{lem1} we conclude 
$$\widetilde{\frak{S}}_{X_\bull}(r)= C^\bulle \to \tau_{\le r}R\ve_*\frak{S}_{X_\bull}(r)_{\et}=\frak{S}_{X_\bull}(r)_{\Nis}$$ 
is a quasi-isomorphism as desired.

\end{appendix}

\bigskip

\bibliographystyle{plain}

\renewcommand\refname{References}

\end{document}